\pdfoutput=1

\documentclass{amsart}
\usepackage[utf8]{inputenc}
\usepackage[T1]{fontenc}

\usepackage{amssymb}
\usepackage{braket}
\usepackage{mathrsfs}
\usepackage{ifthen}
\usepackage{here}
\usepackage{todonotes}
\usepackage{tikz}
\usetikzlibrary{patterns,decorations.pathreplacing,calligraphy}
\usepackage{comment} 
\usepackage{mleftright}
\usepackage[pagebackref,hypertexnames=false]{hyperref} 
\usepackage{cleveref}
 \usepackage[all]{xy}
 \usepackage{amscd}
 \usepackage[alphabetic,backrefs,msc-links]{amsrefs}
 \usepackage{color}
 \usepackage{enumitem}
\newlist{steps}{enumerate}{1}
\setlist[steps, 1]{label = Step \arabic*:}
\usepackage[abs]{overpic}
\usepackage{tikz-cd}
\usepackage{pinlabel}
\usepackage{amsmath, amsthm,verbatim,amsfonts, graphicx}

\usepackage{geometry}
\makeatletter
\DeclareRobustCommand\widecheck[1]{{\mathpalette\@widecheck{#1}}}
\def\@widecheck#1#2{%
   \setbox\z@\hbox{\m@th$#1#2$}%
   \setbox\tw@\hbox{\m@th$#1%
      {%
         \vrule\@width\z@\@height\ht\z@
         \vrule\@height\z@\@width\wd\z@}$}%
   \dp\tw@-\ht\z@
   \@tempdima\ht\z@ \advance\@tempdima2\ht\tw@ \divide\@tempdima\thr@@
   \setbox\tw@\hbox{%
      \raise\@tempdima\hbox{\scalebox{1}[-1]{\lower\@tempdima\box\tw@}}}%
   {\ooalign{\box\tw@ \cr \box\z@}}}
\makeatother

\theoremstyle{plain}
\newtheorem{thm}{Theorem}[section]
\crefname{thm}{Theorem}{Theorems}
\Crefname{thm}{Theorem}{Theorems}
\newtheorem{prop}[thm]{Proposition}
\crefname{prop}{Proposition}{Propositions}
\Crefname{prop}{Proposition}{Propositions}
\newtheorem{lem}[thm]{Lemma}
\crefname{lem}{Lemma}{Lemmas}
\Crefname{lem}{Lemma}{Lemmas}
\newtheorem{cor}[thm]{Corollary}
\crefname{cor}{Corollary}{Corollaries}
\Crefname{cor}{Corollary}{Corollaries}
\newtheorem{rem}[thm]{Remark}
\crefname{rem}{Remark}{Remarks}
\Crefname{rem}{Remark}{Remarks}

\crefname{claim}{Claim}{Claims}
\Crefname{claim}{Claim}{Claims}

\crefname{property}{Property}{Properties}
\Crefname{property}{Property}{Properties}

\crefname{problem}{Problem}{Problems}
\Crefname{problem}{Problem}{Problems}

\crefname{conjecture}{Conjecture}{Conjecture}
\Crefname{conjecture}{Conjecture}{Conjecture}

\theoremstyle{definition}
\newtheorem{defn}[thm]{Definition}
\crefname{defn}{Definition}{Definitions}
\Crefname{defn}{Definition}{Definitions}

\crefname{notation}{Notation}{Notations}
\Crefname{notation}{Notation}{Notations}

\crefname{convention}{Convention}{Conventions}
\Crefname{convention}{Convention}{Conventions}
\crefname{cond}{Condition}{Conditions}
\Crefname{cond}{Condition}{Conditions}

\crefname{assum}{Assumption}{Assumptions}
\Crefname{assum}{Assumption}{Assumptions}

\theoremstyle{remark}
\crefname{rem}{Remark}{Remarks}
\Crefname{rem}{Remark}{Remarks}
\crefname{section}{Section}{Sections}
\Crefname{section}{Section}{Sections}
\crefname{subsection}{Subsection}{Subsections}
\Crefname{subsection}{Subsection}{Subsections}
\crefname{figure}{Figure}{Figures}
\Crefname{figure}{Figure}{Figures}

\newcommand{\Z}{\mathbb{Z}}

\newcommand{\pt}{\mathrm{pt}}

\newcommand{\fraks}{\mathfrak{s}}
\newcommand{\frakt}{\mathfrak{t}}

\newcommand{\inc}{\hookrightarrow}

\newcommand{\C}{\mathbb{C}}

\newcommand{\pr}{\text{pr}}

\newcommand{\R}{\mathbb R}

\newcommand{\F}{\mathbb{F}_2}

\newcommand{\Th}{\mathrm{Th}}

\def\dim{\mathrm{dim}}
\def\rank{\mathrm{rank}}

\def\id{\mathrm{Id}}

\def\ind{\mathrm{ind}}

\def\grad{\mathrm{grad}}

\newcommand{\mbar}[1]{{\ooalign{\hfil#1\hfil\crcr\raise.167ex\hbox{--}}}}

\def\wt{\widetilde}

\def\Ker{\mathrm{Ker}\,}

     \RequirePackage{rotating}                   
    \def\HMt{%
       \setbox0=\hbox{$\widehat{\mathit{HM}}$}
       \setbox1=\hbox{$\mathit{HM}$}
       \dimen0=1.1\ht0
       \advance\dimen0 by 1.17\ht1
       \smash{\mskip2mu\raise\dimen0\rlap{%
          \begin{turn}{180}
              {$\widehat{\phantom{\mathit{HM}}}$}
           \end{turn}} \mskip-2mu    
                \mathit{HM}
                    }{\vphantom{\widehat{\mathit{HM}}}}{}}

\title{%Infinite-order exotic 4-dimensional Dehn twists 
Exotic Dehn twists and homotopy coherent group actions
}

\author{Sungkyung Kang}
\address{Mathematical Institute, University of Oxford, United Kingdom}
\email{sungkyung38@icloud.com}

\author{JungHwan Park}
\address{Department of Mathematical Sciences, KAIST, Republic of Korea}
\email{jungpark0817@kaist.ac.kr}

\author{Masaki Taniguchi} 
\address{Department of Mathematics, Kyoto University, Japan}
\email{taniguchi.masaki.7m@kyoto-u.ac.jp}

\begin{document}

\maketitle

\begin{abstract}
We consider the question of extending a smooth homotopy coherent finite cyclic group action on the boundary of a smooth 4-manifold to its interior. As a result, we prove that Dehn twists along any Seifert homology sphere, except the 3-sphere, on their simply connected positive-definite fillings are infinite order exotic.
\end{abstract}

\section{Introduction}
Dehn twists are fundamental in the study of mapping class groups of surfaces. In this article, we study their generalization in dimension 4. Let $Y$ be a closed oriented 3-manifold with a nontrivial element $\phi \in \pi_1 (\mathrm{Diff}^+(Y))$, based at the identity. Then $\phi$ induces a self-diffeomorphism of $Y \times [0,1] $,
$$ \Phi \colon  Y \times [0,1] \to Y\times [0,1] ; \qquad (s,t) \mapsto (\phi_t(s),t).$$
Given a smooth embedding of $Y \times [0,1]$ into a compact smooth 4-manifold $X$, the self-diffeomorphism $\Phi$ can be extended to a self-diffeomorphism of $X$. The resulting diffeomorphism, $t_X \colon X \to X$, is called a \emph{4-dimensional Dehn twist on $X$ along $Y$}.  Moreover, if $Y$ is the boundary of $X$, then we have a Dehn twist supported in a collar neighborhood of the boundary, and we call such a diffeomorphism a \emph{4-dimensional boundary Dehn twist}. While Dehn twists can be considered in any dimension, we will henceforth omit the dimension, as this article focuses solely on 4-dimensional Dehn twists.

The first nontrivial Dehn twist can be obtained from the work of Baraglia and Konno~\cite{Baraglia-Konno:2022-1}. Their results imply that the boundary Dehn twist on $\mathrm{K3} \smallsetminus \mathring{B}^4$ is not isotopic to the identity rel.\ boundary (see \cite[Section 2]{Kronheimer-Mrowka:2020-1} for more detailed explanations). Furthermore, Kronheimer and Mrowka~\cite[Theorem 1.1]{Kronheimer-Mrowka:2020-1} proved that the Dehn twist on $\mathrm{K3} \# \mathrm{K3}$ along the separating $S^3$ is not isotopic to the identity. Lin~\cite{Lin:2023-1} extended this result and proved that it is not smoothly isotopic to the identity even after a stabilization with $S^2 \times S^2$. These results are obtained by using a one-parameter family of Bauer-Furuta invariants and also the $\mathrm{Pin(2)}$-equivariant version of it.

While the above results are restricted to the case when $Y = S^3$, the general case when $Y \neq S^3$ has garnered significant attention recently (see, e.g., \cite[Section 1.2]{Orson-Powell:2022-1} and \cite{KMT23}). In \cite{KMT23}, the authors study Dehn twists and boundary Dehn twists along Brieskorn spheres using 2-parameter families of Seiberg-Witten theory. Here, the Dehn twist along a Brieskorn sphere $Y$ is induced by the $S^1$-action, which represents a nontrivial element in $\pi_1(\mathrm{Diff}^+(Y))$ (see \cite[Proposition 8.8]{Orson-Powell:2022-1}). They prove that for each nonnegative integer $n$, the boundary Dehn twists on any smooth, compact, simply connected, positive-definite 4-manifold bounded by $\Sigma(2,3,6n+7)$ are not isotopic to the identity rel.\  boundary~\cite[Theorem 1.1]{KMT23}. In particular, since $\Sigma(2,3,13)$ and $\Sigma(2,3,25)$ bound smooth contractible 4-manifolds~\cite{Akbulut-Kirby:1979, Fickle:1984}, the boundary Dehn twists on these contractible smooth 4-manifolds are also not isotopic to the identity rel.\ boundary~\cite[Theorem 1.3]{KMT23} (they also obtain results related to Dehn twists after a stabilization of $S^2 \times S^2$~\cite[Theorem 1.7 and Theorem 1.8]{KMT23} and Dehn twists on closed 4-manifolds~\cite[Theorem 1.10]{KMT23}).

Here, we recall an important property of Dehn twists along $S^3$, which provides the main motivation for this article. Since $\pi_1(\mathrm{Diff}^+(S^3)) \cong \pi_1(\mathrm{SO}(4)) \cong   \mathbb{Z}_2 $, we see that each Dehn twist along $S^3$ has at most order two in the mapping class group. In contrast, for general 3-manifolds, such as Seifert fibered spaces, we may expect them to induce \emph{infinite order} Dehn twists. In this article, we demonstrate that such infinite order Dehn twists exist and can be found in any positive-definite filling of Brieskorn homology spheres:
\begin{thm}\label{thm:main}
    Let $Y \neq S^3$ be a Brieskorn homology sphere, and let $X$ be a smooth, compact, positive-definite 4-manifold filling of $Y$ with $b_1(X) = 0$. Then, the boundary Dehn twist along $Y$  has infinite order in $\pi_0(\mathrm{Diff}^+(X,\partial X))$.\end{thm}

We note that such phenomena cannot be expected for arbitrary smooth fillings of Brieskorn spheres. Indeed, for each Brieskorn sphere, the boundary Dehn twist on its canonical negative-definite resolution is smoothly isotopic to the identity rel.\ boundary~\cite{Orlik-Wagreich:1971-1, Orlik:1972-1}.

Beyond $\Sigma(2,3,13)$ and $\Sigma(2,3,25)$, there are many other  Brieskorn spheres that bound smooth positive-definite 4-manifolds.  For instance, $\Sigma(p, ps \pm 1, ps \pm 2)$ when $p$ is odd, and $\Sigma(p, ps - 1, ps + 1)$ when $p$ is even and $s$ is odd, bound contractible 4-manifolds (see, e.g., \cite{Akbulut-Kirby:1979, Casson-Harer:1981, Fickle:1984} for more examples). Moreover, there are Brieskorn spheres that bound smooth rational homology balls even when they do not bound any smooth contractible 4-manifold, with $\Sigma(2,3,7)$  being the most well-known such example (see, e.g., \cite{Fintushel-Stern:1984, Akbulut-Larson:2018, Oguz:2020} for more examples). In a similar vein, for each positive integer $n$, $\Sigma(2,3,6n+1)$ also bounds a smooth positive-definite 4-manifold, as it can be realized as a $+1$-Dehn surgery on twist knots.

The topological mapping class group of every compact, simply connected, topological 4-manifold has been computed by Orson and Powell \cite{Orson-Powell:2022-1}. In the special case where the compact, simply connected, topological 4-manifold is bounded by a rational homology $S^3$, they show that the topological isotopy type of an orientation-preserving self-homeomorphism fixing the boundary is determined by its induced action on the intersection form (see \cite[Theorem C]{Orson-Powell:2022-1} for the precise statements). Combined with our main theorem we have the following corollary. Here, given a smooth 4-manifolds $X$, we say a diffeomorphism $f \in \mathrm{Diff}^+(X, \partial X)$ is \emph{exotic} if $f$ is isotopic to the identity rel.\ boundary through elements of $\mathrm{Homeo}^+(X, \partial X)$.

\begin{cor}\label{cor:main}
    Let $Y \neq S^3$ be a Brieskorn homology sphere, and let $X$ be a smooth, compact, positive-definite 4-manifold filling of $Y$ with $\pi_1(X) = 0$. Then, the boundary Dehn twist along $Y$ is exotic and has infinite order in $\pi_0(\mathrm{Diff}^+(X,\partial X))$.\qed
\end{cor}

The study of exotic diffeomorphisms was initiated by Ruberman in \cite{Rumberman:1998-1, Rumberman:1999-1}, where he demonstrated the first examples. Recently, this topic has been studied more extensively, leading to many interesting related results \cite{BK18,  Kronheimer-Mrowka:2020-1, IKMT22, Lin:2023-1, KMT23, KM24, KMPW24}. It turns out that among exotic diffeomorphisms, those induced by Dehn twists have the special property of producing loops of homeomorphisms that cannot be extended to loops of diffeomorphisms. To be more precise, given a closed oriented 3-manifold $Y$ bounding a smooth 4-manifold $X$, there is the following fiber sequence:
$$\mathrm{Diff}^+(X, \partial X) \hookrightarrow \mathrm{Diff}^+(X)\to \mathrm{Diff}^+(Y),$$
as well as an analogous fiber sequence for homeomorphism groups (see \cite[Appendix A]{Orson-Powell:2022-1} and \cite[Section 6]{KMT23}). If a boundary Dehn twist on $X$ along $Y$, induced by a nontrivial element $\phi \in \pi_1 (\mathrm{Diff}^+(Y))$, is exotic, then $\phi$ lifts to $\pi_1 (\mathrm{Homeo}^+(X))$ but not to $\pi_1 (\mathrm{Diff}^+(X))$  (see \cite[Section 6]{KMT23} for details). In particular, this implies that if there exists an exotic boundary Dehn twist on $X$ along $Y$, then the inclusion map $\mathrm{Diff}^+(X) \hookrightarrow \mathrm{Homeo}^+(X)$ induces a \emph{non-surjective} map on their fundamental groups. In particular, the results of \cite{Baraglia-Konno:2022-1, Kronheimer-Mrowka:2020-1, Lin:2023-1, KMT23} imply that the corresponding maps for $\mathrm{K3} \smallsetminus \mathring{B}^4$ and $\left(\mathrm{K3} \# S^2 \times S^2\right) \smallsetminus \mathring{B}^4$ are non-surjective, and the same holds for any smooth contractible 4-manifolds bounded by $\Sigma(2,3,13)$ and $\Sigma(2,3,25)$. For closed 4-manifolds, the only known such example is provided by \cite[Theorem 1.3]{Baraglia-Konno:2023-1}, which proved that the corresponding map is non-surjective for $\mathrm{K3}$. 

In this setting, \Cref{cor:main}  implies that for smooth, compact, simply connected, positive-definite 4-manifolds bounded by any Brieskorn homology sphere, the map induced by the inclusion is not only non-surjective but also has infinitely many elements that are not contained in its image:

\begin{cor}\label{cor:cokernel}
    Let $Y \neq S^3$ be a Brieskorn homology sphere, and let $X$ be a smooth, compact, simply connected, positive-definite filling of $Y$. If $i_*$ is the map induced by the inclusion $\mathrm{Diff}^+(X) \hookrightarrow \mathrm{Homeo}^+(X)$ on their fundamental groups, then the $\mathrm{Coker}\, 
     i_*$ contains a subgroup isomorphic to $\mathbb{Z}$. \qed
\end{cor}

Before we lay out the whole strategy of our proof,  we briefly highlight the essence of our result. Specifically, our main theorem is proved by showing that for each positive-definite filling with vanishing $b_1$ of a Brieskorn homology sphere, there is a large enough prime $p$ such that there is no homotopy coherent $\mathbb{Z}_p$-action on the smooth positive-definite filling  that extends the $\mathbb{Z}_p$-action on the boundary, where the boundary action is induced by the $S^1$-action  (see \Cref{sec:homotopycoherentactions,Dehn twists and group actions} for a more detailed explanation of the relationship between homotopy coherent actions and Dehn twists). In particular, it implies that for each positive-definite filling with vanishing $b_1$, there is no $\mathbb{Z}_p$-action that extends the $\mathbb{Z}_p$-action on the boundary, as long as $p$ is large enough. We note that this phenomenon, and indeed more general results regarding the extendability of $\mathbb{Z}_p$-actions, was observed earlier. Anvari and Hambleton~\cite{Anvari-Hambleton:2016-1, Anvari-Hambleton:2021-1} proved that the $\mathbb{Z}_p$-action induced by the $S^1$-action on $\Sigma(a_1,a_2,a_3)$ do not extend to any smooth contractible filling (see also \cite{Kwasik-Lawson:1993-1}). Baraglia and Hekmati~\cite{Baraglia-Hekmati:2022-1, Baraglia-Hekmati:2024-1} extended this result to smooth positive-definite fillings with vanishing $b_1$ of $\Sigma(a_1, a_2, \ldots, a_r)$ under mild assumptions, implying that for a fixed Brieskorn homology sphere, the nonextension result holds for all but finitely many primes (see \cite[Corollary 1.7 and Corollary 1.14]{Baraglia-Hekmati:2022-1} for precise statements). In fact, the work of Baraglia and Hekmati is crucially used in our proof, as we will see below.

Here we summarize the strategy of proof: 
\begin{itemize}
    \item Assuming that the (some finite power of) Dehn twist is smoothly isotopic to the identity rel.\ boundary, we prove that for a sufficiently large prime number \(p\), there is a homotopy coherent action extending an honest Seifert \(\mathbb{Z}_p\)-action on the boundary. Moreover, we can get a family 
\[
(X, \fraks)\to {E} \to B \Z_p
\]
of spin$^c$ 4-manifolds with boundary whose restriction to the boundary is $\Z_p$-Borel construction of Seifert action. This is the main geometric construction in this paper, which is done in \cref{sec:homotopycoherentactions} and \cref{Dehn twists and group actions}. 
\item Next, we develop Bauer–Furuta invariants for the family ${E}$ over $B \Z_p$, yielding a map
\[
{\bf BF}_{E}: Th^f(F_0)  \to Th^f(F_1) \wedge_{B\Z_p} SWF(Y)\times_{\Z_p} E \Z_p
\]
for some $S^1$-equivariant vector bundles $F_0$ and $F_1$ over $B\Z_p$, where $\wedge_{B \Z_p}$ denotes the fiberwise wedge product, $Th^f$ represents the fiberwise Thom space over $B \Z_p$, and $SWF(Y)$ is to Baraglia–Hekmati's $(S^1 \times \Z_p)$-equivariant Seiberg–Witten Floer homotopy type, based on Manolescu’s Seiberg–Witten Floer homotopy type \cite{Ma03}. To be more precise, we need to restrict the family $E $ to a finite skeleton of $B\Z_p$ for the compactness of Seiberg--Witten moduli spaces.
This invariant is constructed in \cref{Bauer--Furuta invariant for homotopy coherent finite group actions}. 

\item 
Applying cohomlogical argument to ${\bf BF}_{E}$, we obtain the family Fr\o yshov inequality: 
\[
\frac{1}{8} \left( c_1^2(\fraks) - b_2(X) \right) \geq \delta^{(p)}_\infty
(Y), 
\]
where $X$ is a smooth, compact, positive-definite filling of $Y$ with $b_1(X) = 0$, and  $\delta^{(p)}_\infty$ is the equivariant delta invariant defined by Baraglia and Hekmati in \cite{Baraglia-Hekmati:2024-1}. 
This Fr\o yshov type inequality will be proven in \cref{Fryshov type inequality from homotopy coherent action}. The inequality can be regarded as a homotopy coherent version of the equivariant Fr\o yshov inequality proved in \cite{Baraglia-Hekmati:2024-1}. 

\item The family Fr\o yshov inequality combined with a theorem of Elkies gives us $0 \geq \delta^{(p)}_\infty(Y)$. On the other hand, computations from \cite[Theorem 1.8]{Baraglia-Hekmati:2022-1} imply that for a large enough prime $p$:
\[
0< \delta^{(p)}_\infty
(Y),
\]
which completes the proof. This will be observed in \cref{Proof of main result}. 
\end{itemize}

The gauge theoretic parts in this paper are based on a certain family Bauer--Furuta invariant and a generalization of Donaldson's diagonalization theorem in a family setting. For these backgrounds, see \cite{Ba21, Baraglia-Hekmati:2024-1, Baraglia-Konno:2022-1, KT22, Ba24}. This idea of homotopy coherent group action appears in \cite{HLS16} in Heegaard Floer theory to develop $G$-equivariant Heegaard Floer homology. A part of our strategy can be thought as a Floer homotopy analog of their construction. 

\subsection*{Acknowledgements} 
We would like to thank Patrick Orson, Mark Powell, and Hokuto Konno for constructive conversations. The second author is partially supported by Samsung Science and Technology Foundation (SSTF-BA2102-02) and the POSCO TJ Park Science Fellowship.
The third author was partially supported
by JSPS KAKENHI Grant Number 20K22319, 22K13921.

In the process of writing this article, we have learned that Konno, Lin, Mukherjee, and Mu\~noz--Ech\'aniz \cite{JLME24} independently obtained a result related to \Cref{thm:main}. For instance, they show that boundary Dehn twists on certain symplectic fillings and on a specific family of contractible fillings have infinite order, using a completely different technique. Also, the authors have been informed that Jin Miyazawa also obtained a result related to Dehn twists on the Milnor fibers using a different technique. 

%Additionally, the authors have been informed that Konno and Sasahira are working on a family version of the Fr\o yshov invariant and an inequality in a more general context.

\section{Homotopy coherent finite group actions}\label{sec:homotopycoherentactions}

We begin by recalling the Borel construction for $\mathbb{Z}_p$-actions on a manifold. Let $X$ be a smooth manifold, possibly with boundary, equipped with a smooth $\mathbb{Z}_p$-action, where $p$ is an odd prime. We can then define a smooth $X$-bundle (i.e.\ a $X$-bundle with structure group $\mathrm{Diff}^+(X)$) as follows:
\[
X \rightarrow (X\times E\mathbb{Z}_p)/\mathbb{Z}_p \rightarrow B\mathbb{Z}_p.
\]
While this is probably the simplest definition of Borel construction, we will provide a way to construct it by gluing trivial $X$-bundles over simplices, with a view towards applying it for Dehn twists.

The Eilenberg-MacLane space $B\mathbb{Z}_p$ admits a canonical structure of a CW complex, defined as follows. We choose a generator of $\mathbb{Z}_p$, which we denote by $[1]$. Then the elements of $\mathbb{Z}_p$ are given by $[0],[1],\dots,[p-1]$. Then we build up a CW complex in the following way; note that this is the geometric realization of the nerve of the 1-object category describing the group $\Z_p$.
\begin{itemize}
    \item We have one 0-cell, which we denote by $\ast$;
    \item We have $p$ 1-cells, labelled by elements of $\mathbb{Z}_p$, so $[0],\dots,[p-1]$;
    \item For each triple $([i], [j], [i+j])$, we glue in a 2-simplex;
    \item Inductively, for every $n$ and any sequence $[i_1],\dots,[i_n]$ of elements in $\mathbb{Z}_p$ with $(n-1)$-simplexes spanned by the following length $n-1$ sequences:
    \[
    ([i_1],\cdots,[i_{n-1}]);\quad ([i_1],\cdots,[i_{j-1}],[i_j + i_{j+1}\text{ mod }p],[i_{j+2}],\cdots,[i_n]) \text{ for all } 1 \le j \le n-1;\quad ([i_2],\cdots,[i_n]).
    \]
    We glue in an $n$-simplex so that its $n+1$ facets are given by the $(n-1)$-simplexes given as above.
\end{itemize}
%{{\color{red} I think this is a geometric realization of some simplicial set. Is there a reference for it? \color{blue} It definitely is. But I thought this is very standard so probably did not require any reference... Can I just say this is the nerve of the category $\Z_p$?}
To constructively define the Borel construction for $X$, we proceed as follows. We start with $X$, which is the trivial $X$-bundle over the 0-skeleton $\ast$. For each $[i]\in \mathbb{Z}_p$, we glue the trivial $X$-bundle over the 1-simplex $[i]$ to $X$, by gluing its initial point to $X$ via identity and terminal point to $X$ via the action of $[i]$. Then, for each 2-simplex to be glued in, its boundary triangle looks as follows:
\[
\xymatrix{
& \ast\ar[rd]^{[j]} & \\
\ast\ar[ru]^{[i]}\ar[rr]^{[i+j\text{ mod }p]} && \ast.
}
\]
Observe that the $X$-bundle over the upper two edges is the same as the $X$-bundle over the lower edge; they are both induced by the action of $[i+j\text{ mod }p]$. Therefore the trivial $X$-bundle over a 2-simplex can be canonically glued in. The same argument can then be inductively used to glue trivial $X$-bundles over $n$-simplices.

\begin{rem}
    This CW complex structure on $B\mathbb{Z}_p$ will be used in later sections. In particular, for each $n \geq 0$, we denote the $n$-skeleton with respect to the CW complex structure described above by $(B\mathbb{Z}_p)_n$.
\end{rem}

This Borel construction suggestes that, instead of studying smooth $\mathbb{Z}_p$-actions on $X$, we might instead study smooth $X$-bundles over $B\mathbb{Z}_p$, which corresponds to homotopy classes of maps $B\mathbb{Z}_p\rightarrow B\mathrm{Diff}(X)$. This is precisely the definition of smooth homotopy coherent actions of $\mathbb{Z}_p$ on $X$. For the sake of self-containedness, we will make this a separate definition here.

\begin{defn}
    Let $X$ be a smooth manifold and $G$ be a discrete group. A \emph{smooth homotopy coherent $G$-action} on $X$ is a homotopy class of maps $BG\rightarrow B\mathrm{Diff}(X)$. If $X$ has a boundary, then a smooth homotopy coherent $G$-action $\bf H$ on $X$ restricts to a smooth homotopy coherent $G$-action ${\bf H}\vert_{\partial X}$ on $\partial X$; we refer to it as the restriction of $\bf H$ to the boundary $\partial X$.
\end{defn}

    From now on, we will regard strict $\mathbb{Z}_p$-actions automatically as homotopy coherent $\mathbb{Z}_p$-actions.

\begin{defn}
    Given a homotopy coherent $G$-action ${\bf H} \colon BG\rightarrow B\mathrm{Diff}(X)$, we refer to the induced map
    \[
    G\cong \pi_1(BG) \xrightarrow{{\bf H}_\ast} \pi_1(B\mathrm{Diff}(X)) \cong \pi_0(\mathrm{Diff}(X)) = \mathrm{MCG}(X)
    \]
    the \emph{homotopy-monodromy} of $\bf H$. We sometimes also compose it with actions of $\mathrm{MCG}(X)$ on homology or cohomology of $X$; we denote them as homology monodromy or cohomology monodromy, respectively.
\end{defn}

It is clear that any smooth $\mathbb{Z}_p$-action on $X$ uniquely specifies a smooth homotopy coherent $\mathbb{Z}_p$-action on $X$. However, in some cases, it is possible to construct a smooth homotopy coherent $\mathbb{Z}_p$-action directly, by gluing trivial $X$-bundles over simplices; for those examples, it is unclear whether they should be induced from an honest smooth $\mathbb{Z}_p$-action. This will be discussed again in \Cref{Dehn twists and group actions}.

In order to set the families Seiberg--Witten theory, we need a homotopy coherent spin$^c$ structure. Let us denote the isomorphism class of a spin$^c$ structure $\fraks$ by $[\fraks]$. From now on, we consider compact-oriented $4$-manifold $X$ with boundary $Y$. We further suppose $Y$ is a homology $3$-sphere. 

\begin{lem}\label{S1}
    Let $G$ be a discrete group and ${\bf H}$ be a homotopy coherent smooth $G$-action on $X$ with $b_1(X)=0$. Suppose that its homotopy-monodromy preserves the orientation of $X$ and an isomorphism class  of a $\mathrm{spin}^c$ structure $\mathfrak{s}$ on $X$. Then the corresponding map $BG \to B\mathrm{Diff} (X)$ factors through 
    \[
    BG\to B\mathrm{Diff}^+  (X; [\fraks]), 
    \]
    where $\mathrm{Diff}^+  (X; [\fraks])$ is the group of orientation-preserving diffeomorphisms which preserve $[\fraks]$. 
    Moreover, if there is a lift of the $(i-1)$-skeleton, then the obstruction to extending to the $i$-skeleton of the lifting problem:
    \[
    \xymatrix{
     & B\mathrm{Aut}(X;\fraks) \ar[d] \\
     BG \ar@{.>}[ur]^{\exists ?} \ar[r] & B\mathrm{Diff}^+(X,[\fraks])
    }
    \]
lies in 
\[
o({\bf H})\in  H^i \left(B G; \pi_{i-1} (BS^1)\right), 
\]
where $\mathrm{Aut} (X; \fraks)$ denotes the set of all isomorphisms of the spin$^c$ structure $\fraks$ as the principal bundle. 
The same is true for a family of 3-manifolds. 
\end{lem}
\begin{proof}
The first claim is obvious. 
Note that we have an exact sequence of topological groups: 
\[
\mathrm{Map} (X, U(1)) \to \mathrm{Aut} (X; \fraks) \to \mathrm{Diff}^+ (X; [\fraks])
\]
which induces a fibration
\[
B\mathrm{Map} (X, U(1)) \to B\mathrm{Aut} (X; \fraks) \to B\mathrm{Diff}^+ (X; [\fraks]).
\]
In order to compute homotopy type of $\mathrm{Map}(X, U(1))$, we consider the fibration
\[
\mathrm{Map}^0(X, U(1))\to \mathrm{Map}(X, U(1)) \to S^1, 
\]
where $\mathrm{Map}^0(X, U(1))$ denotes the group of gauge transformations with $g(x_0) =\id$ for a some fixed base point $x_0 \in \partial X \subset X$. We claim $\mathrm{Map}^0(X, U(1))$ is contractible hence the evaluation map $\mathrm{Map}(X, U(1)) \to S^1$ is weak homotopy equivalence. 

Consider the covering map $p\colon \mathbb{R}\rightarrow U(1)$ with $p(0)=\mathrm{id}$. Using the assumption that $H^1(X;\mathbb{Z})=0$, we know that every map $f\in \mathrm{Map}^0(X,U(1))$ has a unique lift $\tilde{f}\in \mathrm{Map}^0(X,\mathbb{R})$. Here, $\mathrm{Map}^0(X,\mathbb{R})$ is the space of maps $g:X\rightarrow\mathbb{R}$ with $g(x_0)=0$.

Now we have continuous maps
\[
\begin{split}
& \mathrm{Map}^0(X,U(1)) \xrightarrow{f\mapsto \tilde{f}} \mathrm{Map}^0(X,\mathbb{R}) \\
& \mathrm{Map}^0(X,\mathbb{R}) \xrightarrow{g\mapsto p\circ g} \mathrm{Map}^0(X,U(1))
\end{split}
\]
which are inverses to each other by the uniqueness of lift. Therefore they are homeomorphisms. Since $\mathrm{Map}^0(X,\mathbb{R})$ is clearly contractible, we deduce that $\mathrm{Map}^0(X,U(1))$ is also contractible. Then the desired statement follows from the general obstruction theory. 
\end{proof}

 From the naturality of the obstruction class, we see the following: 
 \begin{cor}\label{existence of lift}
Let $G$ be a discrete group and ${\bf H}$ be a homotopy coherent smooth $G$-action on $X$ with $b_1(X)=0$. Suppose that its homotopy-monodromy preserves the orientation of $X$ and an isomorphism class  of a $\mathrm{spin}^c$ structure $\mathfrak{s}$ on $X$. If the homotopy coherent action ${\bf H}$ comes from a honest $G$-action on $Y$ and the natural inclusion map 
\[
Y \to X
\]
induces isomorphism on their cohomologies, then there is a lift 
 \[
    \xymatrix{
     & B\mathrm{Aut}(X;\fraks) \ar[d] \\
     BG \ar[ur]\ar[r] & B\mathrm{Diff}^+(X,[\fraks]).
    }
    \]
 \end{cor}
 \begin{proof}
     If we restrict the lifting problem to the boundary $Y$, we have
   \[
    \xymatrix{
     & B\mathrm{Aut}(Y;\fraks|_{Y}) \ar[d] \\
     BG \ar[ur]\ar[r] & B\mathrm{Diff}^+(Y,[\fraks|_{Y}]).
    }
    \]
Since the $G$-action on $Y$ is strict so we can consider the Borel construction for $(Y, \fraks|_{Y})$, and the corresponding obstruction 
\[
o ( {\bf H}|_{Y}
 ) \in  H^i \left(B G; \pi_{i-1} ( B\mathrm{Map} (Y, U(1)))\right)
 \]
 vanishes for each $i$. Moreover, since the obstruction class is natural, we have 
 \[
 i^* o({\bf H} ) = o( {\bf H}|_Y)=0
 \]
 and the natural restriction map 
 \[
 i^* \colon H^i\left(B G; \pi_{i-1} ( B\mathrm{Map} (X, U(1))\right) \to H^i\left(B G; \pi_{i-1} ( B\mathrm{Map} (Y, U(1)))\right)
 \]
 is isomorphism from
 \[
 \pi_{i-1} (BS^1) \cong \pi_{i-1} ( B\mathrm{Map} (X, U(1))) \xrightarrow{\cong} \pi_{i-1} ( B\mathrm{Map} (Y, U(1))) \cong \pi_{i-1} (BS^1) 
 \]
which is observed in \cref{S1}, where the map is induced from the inclusion $Y \hookrightarrow X$.
 This completes the proof of $o({\bf H} ) =0$ for every skeleton.
 \end{proof}
 \begin{rem}
 In the above computation, the only obstruction lies in $H^3 (B G; \pi_{2} ( BS^1)) \cong H^3(BG; \Z) $. 
     In this paper, we do not care about choices of lifts since we will not consider invariants of homotopy coherent actions. 
 \end{rem}

\section{Dehn twists and group actions}\label{Dehn twists and group actions}

Let $Y$ be a closed oriented 3-manifold with an $S^1$-action that induces a nontrivial element $\phi \in \pi_1(\mathrm{Diff}^+(Y))$, based at the identity. Specifically, the smooth $S^1$-action on $Y$ is given by $e^{2\pi it} \cdot s := \phi_t(s)$, where $e^{2\pi it} \in S^1$ and $s \in Y$. Let $\Phi$ be the induced diffeomorphism of $Y \times [0,1]$, defined by
$$ \Phi \colon  Y \times [0,1] \to Y\times [0,1] ; \qquad (s,t) \mapsto (\phi_t(s),t).$$
Thus, if $Y$ bounds a smooth, compact 4-manifold $X$, then there is a boundary Dehn twist on $X$ along $Y$. Note that  they are only well-defined up to isotopy rel.\ boundary. In other words, we treat a boundary Dehn twist as an isotopy class of diffeomorphisms rel.\ boundary  rather than a specific concrete diffeomorphism.

\begin{lem} \label{lem:commutationoncylinder}
    Let $Y$ be a closed, oriented 3-manifold with an $S^1$-action that induces a $\mathbb{Z}_p$-action, and let $X$ be a smooth, compact 4-manifold that bounds $Y$. If the boundary Dehn twist $\Phi$ is smoothly isotopic rel.\ boundary to the identity, then there exists a diffeomorphism $\tau$ of $X$, extending the action of the generator of $\mathbb{Z}_p$ on $Y$, and a smooth isotopy ${H_t}$ from $\Phi$ to the identity, such that $\tau^p=\Phi$ and $H_t \tau = \tau H_t$ for each $t$.
\end{lem}

% \begin{lem} \label{lem:commutationoncylinder}
%     Let $Y$ be a closed, oriented 3-manifold with an $S^1$-action that induces a $\mathbb{Z}_p$-action, and let $X$ be a smooth, compact 4-manifold that bounds $Y$. If the boundary Dehn twist $\Phi$ is smoothly isotopic rel.\ boundary to the identity, then there exists a diffeomorphism $\tau$ of $X$, extending the action of the generator of $\mathbb{Z}_p$ on $Y\times \{0, 1\}$ and satisfying $\tau^p = \Phi$, and a smooth isotopy $\{H_t\}$ of $\Phi$ to the identity, satisfying $H_t \tau = \tau H_t$ for each $t$. 

% \end{lem}
\begin{proof}
    Choose a collar neighborhood $Y \times [-1, 1] \subset X$ such that the boundary of $X$ is identified with $Y \times \{1\}$. Also, choose a smooth monotone function $f\colon \mathbb{R} \rightarrow [0, 2\pi]$ such that $f(x) = 0$ for $x \leq \frac{1}{3}$ and $f(x) = 2\pi$ for $x \geq \frac{2}{3}$. Define a self-diffeomorphism $\tau$ of $Y \times [-1, 1]$ as follows, and extend it by the identity outside $Y \times [-1, 1]$:
    \[
    \tau(y,s) = \left( e^{\frac{if(s)}{p}} \cdot y,s \right).
    \]
    Then $\tau^p$ is the boundary Dehn twist $\Phi$. More precisely, since the Dehn twist is defined up to isotopy rel.\ boundary, we may choose $\tau^p$ as its representative. Thus, we can assume that the diffeomorphism $\Phi$ is defined as: $$\Phi(y,s) = \left( e^{if(s)}\cdot y,s \right).$$ We further consider a shifted Dehn twist $\Phi^\prime$, which is defined on $Y\times [-1,1]$ by 
    \[
    \Phi^\prime(y,s) = \left( e^{if(s+1)} \cdot y,s \right).
    \]
    This is clearly isotopic to $\Phi$ rel.\ boundary; we can choose the isotopy ${H^\prime_t}$ from $\Phi$ to $\Phi^\prime$ as follows.
    \[
    H^\prime_t(x) = 
    \begin{cases}
        \left(e^{i f(s + g(t))} \cdot y, s\right) & \text{if } x=(y,s) \in Y \times [-1,1], \\
        \, x & \text{otherwise.}
    \end{cases}
\]
where $g\colon \mathbb{R}\rightarrow [0,1]$ is a smooth monotone function such that $g(t)=0$ for $t\le \frac{1}{3}$ and $g(t)=1$ for $g\ge \frac{2}{3}$. Then it is clear that $H^\prime_t \tau = \tau H^\prime_t$ for each $t \in [0,1]$, since both sides are the identity outside $Y \times [-1,1]$, and inside $Y \times [-1,1]$, we have:
    \[
    H^\prime_t \tau (y,s) = \tau H^\prime_t (y,s) = \left( e^{i\left(f(s+g(t))+\frac{f(s)}{p}\right)} \cdot y,s \right).
    \]

    Now observe that $X \smallsetminus (Y \times (0,1])$ is diffeomorphic to $X$, and in this `smaller' $X$, $\Phi^\prime$ is the boundary Dehn twist. Hence, by assumption, $\Phi^\prime$ is smoothly isotopic rel.\ boundary (which is $Y \times {0}$ here) to the identity. Choose such an isotopy ${H''_t}$ and extend each $H''_t$ by the identity on $Y \times [0,1]$, so that it becomes an isotopy of diffeomorphisms on the entire $X$ rel.\ $Y \times [0,1]$, from $\Phi^\prime$ to the identity. Since $\tau$ is identity outside $Y\times [0,1]$, we have
    \[
    H''_t \tau = \tau H''_t \quad \text{ for each }t\in [0,1].
    \]
    Therefore, the concatenation of the isotopies $H'_t$ and $H''_t$ yields the desired homotopy $H_t$ from $\Phi$ to the identity.
\end{proof}

\begin{rem}
    Note that, in \Cref{lem:commutationoncylinder} the relation $H_t \tau = \tau H_t$ is to be satisfied strictly. We are not merely claiming that they are isotopic diffeomorphisms. 
\end{rem}

\begin{lem}\label{lem:extensiononcylinder}
    Under the same assumption as in \Cref{lem:commutationoncylinder}, there exists a smooth homotopy coherent $\mathbb{Z}_p$-action on $X$ which extends the given $\mathbb{Z}_p$-action on $Y$.
\end{lem}
\begin{proof}
    We use the CW complex structure on $B\mathbb{Z}_p$ as discussed in \Cref{sec:homotopycoherentactions}, and we apply \Cref{lem:commutationoncylinder} to obtain a diffeomorphism $\tau$ and an isotopy ${H_t}$ satisfying $\tau^p = \Phi$ and $\tau H_t = H_t \tau$ for each $t$. We start with the trivial $X$-bundle $B_\ast$ over the 0-cell $\ast$; we glue in trivial bundles $B_{[i]}$ over the 1-cell labelled by $[i]$ using the diffeomorphism $\tau^i$. Now, if $[i],[j]\in \mathbb{Z}_p$ and $i+j\ge p$, filling the ($X$-bundle over the) triangle
    \[
    \xymatrix{
    & \ast\ar[rd]^{[j]} & \\
    \ast\ar[ru]^{[i]}\ar[rr]^{[i+j\text{ mod }p]} && \ast
    }
    \]
    amounts to choosing an isotopy from $\tau^{i+j}$ to $\tau^{i+j-p}$; we choose $\{H_t \tau^{i+j-p}\}$.

    But then, gluing trivial $X$-bundles over 3-simplices becomes a nontrivial problem. Given three elements $[i],[j],[k]\in \mathbb{Z}_p$, we have to find an isotopy (of isotopies) between the concatenation of the following two isotopies (we assume the case $i+j\ge p$, $j+k\ge p$, $i+j+k\ge 2p$, as other cases are similar):
    \[
    \tau^{i+j+k} \xrightarrow{\{ H_t \tau^{i+j+k-p} \}} \tau^{i+j+k-p} \xrightarrow{\{ H_t \tau^{i+j+k-2p} \}} \tau^{i+j+k-2p},
    \]
    and
    \[
    \tau^{i+j+k} \xrightarrow{\{ \tau^i H_t \tau^{j+k-p} \}} \tau^{i+j+k-p} \xrightarrow{\{ H_t \tau^{i+j+k-2p} \}} \tau^{i+j+k-2p}.
    \]
    However, since $H_t \tau = \tau H_t$ for each $t$, these two isotopies are exactly the same, and thus choosing the constant isotopy between them gives a way to glue trivial $X$-bundles over 3-simplices. The higher simplices can now be glued in the same way.
\end{proof}

Now, we observe the monodromy of the homotopy coherent $\Z_p$-action constructed above. 
\begin{lem} \label{lem:monodromy_trivial}
    Let $\mathcal{E}$ be the $Y$-bundle and $\wt{\mathcal{E}}$ the $X$-bundle over $B\mathbb{Z}_p$, induced by the $\mathbb{Z}_p$-action on $Y$ and the homotopy coherent $\mathbb{Z}_p$-action on $X$, respectively, as discussed in \Cref{lem:extensiononcylinder}. Then, their homotopy-monodromies
    \[
    m_\mathcal{E}\colon \Z_p\rightarrow \mathrm{MCG}(Y)\qquad\text{ and }\qquad m_{\wt{\mathcal{E}}}\colon \Z_p\rightarrow \mathrm{MCG}(X)
    \]
    are both trivial.
\end{lem}
\begin{proof}
    Since $m_\mathcal{E}$ and $m_{\wt{\mathcal{E}}}$ depend only on the restrictions of $\mathcal{E}$ and $\wt{\mathcal{E}}$ to a simple closed curve $\gamma \subset B\mathbb{Z}_p$ that generates $\pi_1(B\mathbb{Z}_p) \cong \mathbb{Z}_p$, it is straightforward to see that $m_\mathcal{E}$ is determined by the action of $\mathbb{Z}_p$ on $Y$. As this $\mathbb{Z}p$ action on $Y$ extends to the $S^1$-action, it follows that $m_\mathcal{E}$ is trivial.

    For $m_{\wt{\mathcal{E}}}$, we observe that in the proof of \Cref{lem:extensiononcylinder}, the commutation between $\tau$ and the isotopy $\{H_t\}$ was used only to attach $X$-bundles over 2-cells and higher cells to a $X$-bundle over the 1-skeleton of $B\Z_p$. Hence we see that the image of the generator of $\Z_p$ by $m_{\wt{\mathcal{E}}}$ is given by the isotopy class of $\tau$ as defined in the proof of \Cref{lem:commutationoncylinder}. Since $\tau$ is clearly isotopic to identity (since we are not considering isotopies rel.\ boundary), it follows that $m_{\wt{\mathcal{E}}}$ is also trivial.
\end{proof}

 %    Let $X$ be a smooth manifold endowed with a smooth $S^1$-action, $Y$ be a smooth manifold with the same dimension as $X$, and $W$ be a smooth cobordism between $X$ and $Y$. Consider the (strict) action of $G$ on $\partial W$, where the action on $Y$ is trivial and the action on $X$ is defined by the inclusion $G \subset S^1$. If the Dehn twist $T_X$ on a collar neighborhood of $X$ is smoothly isotopic to the identity in $W$, then the given $G$-action on $\partial W$ extends smoothly to a homotopy coherent $G$-action on $W$.

 % Moreover, if X is a spin$^c$ 4-manifold, the given $G$-action on $\partial W$ extends smoothly to a homotopy coherent $G$-spin$^c$ action on $W$. 
 
Summarizing \Cref{lem:commutationoncylinder,lem:extensiononcylinder,lem:monodromy_trivial}, we obtain the following proposition.

\begin{prop}\label{extensiononcoborodism}

    Let $Y$ be a closed, oriented 3-manifold with an $S^1$-action that induces a $\mathbb{Z}_p$-action, and let $X$ be a smooth, compact 4-manifold that bounds $Y$. If the boundary Dehn twist $\Phi$ is smoothly isotopic rel.\ boundary to the identity, then the $\mathbb{Z}_p$-action on $Y$ extends to a smooth homotopy coherent $\mathbb{Z}_p$-action on $X$. 
    Moreover, the homotopy-monodromy of this action is trivial. \qed
    
    %Moreover, if $W$ is a spin$^c$ 4-manifold, then the homotopy coherent $G$-action can be made compatible with the spin$^c$ structure.
    
    % given $G$-action on $\partial W$ extends smoothly to a homotopy coherent $G$-spin$^c$ action on $W$. 
    % and consider a $G$-action on $Y\times \{0,1\}$, where the action on $Y\times\{0\}$ is trivial and the action on $Y\times \{1\}$ is the free $G$ action induced by the $S^1$-action.  If the induced Dehn twist $\Phi$ is smoothly isotopic rel boundary to the identity, then the  $G$-action extends smoothly to a homotopy coherent $G$-action on $X\times [0,1]$.  Moreover, the homotopy coherent $G$-action can be made compatible with the spin$^c$ structure.
    
\end{prop}

\begin{rem} \label{rem:higherpower}
    The above lemma continues to hold when the condition that the Dehn twist is smoothly isotopic to rel.\ boundary is replaced by a condition that $d$-th power of the Dehn twist is smoothly isotopic to identity rel.\ boundary for some integer $d$ relatively prime to $p$. To prove this, one simply has to replace $d$ by its multiple $d^\prime$ satisfying $d^\prime \equiv 1 \pmod p$ and replace the $\tau$ by its $d^\prime$th power so that its $p$-th power is $\Phi^{d'}$.
\end{rem}

\section{Bauer--Furuta invariant for homotopy coherent finite group actions}\label{Bauer--Furuta invariant for homotopy coherent finite group actions}

\subsection{Homotopy coherent Bauer--Furuta invariant for 4-manifolds with boundary} \label{subsec:BF_boundary}

Let $p$ be an odd prime and $n$ be a positive integer. We fix the following geometric data: 

\begin{itemize}
    \item A compact oriented spin$^c$ 4-manifold $(X, \mathfrak{s})$ with a spin$^c$ rational homology 3-sphere boundary $(Y, \mathfrak{t})$ such that 
    \[
    b_1(X) =0 \qquad\text{ and }\qquad b_2^+(X)=0, 
    \]
    and moreover, assume that the $\Z_p$-action on $Y$ preserves the isomorphism class of $\mathfrak{t}$.
    \item A map $B\Z_p \to B\mathrm{Diff}^+ (X ; [\fraks])$ defines a family 
    \[
  (X, \mathfrak{s}) \to E \to B\Z_p.   
    \]
    We assume that the induced family of $(Y,\mathfrak{t})$ from $E$ is isomorphic to the Borel construction of a strict $\Z_p$-action.
\end{itemize}
For these input data, we shall construct a Bauer--Furuta type invariant. In order to describe our main result on the Bauer--Furuta type invariant, we first review Baraglia--Hekmati's $\Z_p$-equivariant Seiberg--Witten Floer homotopy type~\cite{Baraglia-Hekmati:2024-1}.

Fix a $\Z_p$-invariant Riemann metric $g_Y$ on $Y$. It defines a fiberwise Riemann metric on the Borel construction $E\Z_p\times_{\Z_p} Y \to B\Z_p$. Since we know the induced bundle of the 3-manifold bundle as the boundary of $E$ is isomorphic to $E\Z_p\times_{\Z_p} Y \to B\Z_p$, we take a fiberwise Riemann metric on $E$ which coincides with $g_Y$ on the boundary. 
Let us denote $I^\lambda_\mu$ the Baraglia--Hekmati's $(S^1\times \Z_p)$-equivariant Conley index defined in \cite{Baraglia-Hekmati:2024-1}, where $\mu, -\lambda$ are taken to be sufficiently large, which will be reviewed below.

% Let $\s$ also denote a spin$^c$ structure of $Y$ restricted from $X$ and moreover assume that the $G$-action preserves the isomorphism class of $\s$.

% Let $p$ be a prime and $Y$ be a rational homology $3$-sphere.
% Suppose $Y$ is equipped with a spin$^c$ structure $\s$ and a smooth $\Z/p$-action preserving the orientation and the isomorphism class of $\s$.

%Fix a $G$ invariant Riemannian metric $g$ on $Y$.
Choose a reference spin connection $A_0$ such that the associated connection on the determinant line bundle is flat.
As shown in \cite[Section 3.2]{Baraglia-Hekmati:2024-1}, for $g \in \Z_p$, we can choose a lift $\wt{g}\colon S \to S$ to the spinor bundle preserving $A_0$. 
Here we use the assumption that the $G$ action preserves the isomorphism class of $\mathfrak{t}$ and $b_1(Y)=0$.
Let  $G_\mathfrak{t}$ be the set of unitary automorphisms $u\colon S\to S$ preserving $A_0$ and lifting the $\Z_p$ action on $Y$.
Then we have an extension
\[
1\to S^1\to G_{\mathfrak{t}} \to \Z_p \to 1.
\]
This extension is always trivial, as shown in \cite[Section 5]{Baraglia-Hekmati:2024-1}. Therefore, we take a section $G_{\mathfrak{t}}\cong S^1\times \Z_p$.

From the $\Z_p$-action on $Y$, we have an action of $G_{\mathfrak{t}}$ on the global Coulomb slice
\[
V=\ker d^* \oplus \Gamma(S) \subset i\Omega^1(Y) \oplus \Gamma(S)
\]
and a formally self-adjoint linear operator
\[
l\colon V\to V;  \qquad (a, \phi) \mapsto (* da, D_{A_0}), 
\]
where $*$ means the Hodge star operator with respect to $g$ and $D_{A_0}$ means the spin$^c$ Dirac operator with respect the spin$^c$ connection $A_0$.
We take a finite-dimensional approximation $V^\mu_\lambda(g_Y)$ obtained as the direct sum of all eigenspaces of $l$ in the range $(\lambda, \mu]$, again which as acted by $G_{\mathfrak{t}}$. 
The gradient vector field $\operatorname{grad} CSD$ of the Chern--Simons Dirac functional can be written as $\grad CSD = l + c$, where $c$ is a compact map. 

By finite-dimensional approximation of the gradient flow equation with respect to $\operatorname{grad} CSD$, we obtain a $G_{\mathfrak{t}}$-equivariant Conley index 
\[
I^\mu_\lambda(g_Y)
\]
for sufficiently large real numbers $\mu, -\lambda$.
For the details of the construction, see \cite{Ma03, Baraglia-Hekmati:2024-1}.
Now, a metric-dependent equivariant Floer homotopy type is defined as 
\[
SWF(Y, \mathfrak{t}, g_Y):=\Sigma^{-V^0_\lambda(g_Y)}I^\mu_\lambda(g_Y)
\]
in certain equivariant stable homotopy category.

Now, we state our result on the families Seiberg--Witten theory associated to homotopy coherent $\Z_p$-actions. 

\begin{thm}\label{hc Bf inv}
From the above bundle $E$, we can construct a fiberwise $S^1$-equivariant continuous map
\[
 {\bf BF}_{E} \colon  \begin{CD}
     \mathrm{Th}^f(W_0 )  @>>> \mathrm{Th}^f(W_1) \wedge_{ (B\Z_p)_n}  (E \Z_p)_n\times_{\Z_p} I^\mu_\lambda  \\
  @VVV    @VVV \\
     (B\mathbb{Z}_p)_n     @>{\mathrm{id}}>>  (B\mathbb{Z}_p)_n ,  
  \end{CD}
\] where
\begin{itemize} 
\item[(i)] %The base space $(B\mathbb{Z}_p)_n$ is a $n$-skeleton of the classifying space $B\Z_p$.
%\item  $I_{\lambda}^{\mu}$ is the Conley index used to define the Seiberg--Witten Floer homotopy type of $Y$, where $\mu, -\lambda$ are taken to be sufficiently large.
Let $(N,L)$ be an index pair to define the Baraglia--Hekmati's equivariant Conley index $I_{\lambda}^{\mu}$ so that
\[
I_{\lambda}^{\mu} = N/L,
\]
and let $\mathrm{Th}^f$ be the fiberwise Thom space parametrized by $(B\mathbb{Z}_p)_n$. 
\item[(ii)] The vector bundles $W_{0}, W_{1} \to BG$ satisfy that each $W_{i}$ is the direct sum of a real vector bundle $W_{i}(\R)$ and a complex vector bundle $W_{i}(\C)$ over $(B\mathbb{Z}_p)_n$:
\[
W_{i} = W_{i}(\R) \oplus W_{i}(\C).
\]
The $S^1$-actions on $W_{0}(\R)$ and $W_{1}(\R)$ are trivial and the $S^1$-actions on $W_{0}(\C)$ and $W_{1}(\C)$ come from the structure of $\C$-vector spaces. 
    % \item The notation $\mathrm{Th}^f$ means the fiberwise Thom space parametrized by $(B\mathbb{Z}_p)_n$. 
\item[(iii)]  The $S^{1}$-invariant part of the map ${\bf BF}_E$ is fiberwise homotopy equivalent. 
%a fiberwise $S^{1}$-equivariant linear map between vector bundles, also denoted by the same symbol $BF_{{\bf H_\mathfrak{s}}}^{S^{1}}$ :
%\begin{align*}
%BF_{{\bf H_\mathfrak{s}}}^{S^{1}} \colon W_{0}(\R) \to W_{1}(\R) \oplus {{\bf V}_{\lambda}^{0}(\R)} , 
%\end{align*}
%where ${{\bf V}_{\lambda}^{0}(\R)}$ denote the Borel construction $EG\times _G V_{\lambda}^{0}(\R)$.
%Then $p_{2} \circ f(W_{0}(\R)) \subset V_{\lambda}^{\mu}(\R)$ is actually contained in the non-positive eigenvalue part:
%\begin{align}
%\label{eq: nonpositive eigen incl}
%p_{2} \circ f(W_{0}(\R)) \subset V_{\lambda}^{0}(\R).
%\end{align}
%This follows from \cref{linear injection}.
%Let $p_{V_{\lambda}^{0}(\R)} \colon V_{\lambda}^{\mu}(\R) \to V_{\lambda}^{0}(\R)$ be the $L^{2}$-projection.
%Then the map
%\begin{align*}
%\left(\id_{W_{1}(\R)} \oplus p_{V_{\lambda}^{0}(\R)}\right) \circ BF_{{\bf H_\mathfrak{s}}}^{S^{1}} \colon W_{0}(\R) \to W_{1}(\R) \oplus {{\bf V}_{\lambda}^{0}(\R)}
%\end{align*}\footnote{JP:should $\mu$ be $0$ here? MT: Yes I changed }
%is a fiberwise linear isomorphism. 
\item[(iv)] We have
\begin{align*}
\begin{split}
\rank_{\C}W_{0}(\C) - \rank_{\C}W_{1}(\C) 
&= \ind_{\C}{D^{+}_{\hat{A}_{b}}} + \dim_{\C} V^{0}_{\lambda}(\C)\\
&= \frac{c_{1}(\fraks)^{2} - \sigma(X)}{8} + n(Y,\frakt,g_Y)+ \dim_{\C} V^{0}_{\lambda}(\C),  
\end{split}
\end{align*}
where the notation $n(Y,\frakt,g_Y)$ is used with the same meaning as in \cite[The equation (6)]{Ma03}, $V^\mu_\lambda (\R)$ and $V^\mu_\lambda (\C)$ means the eigenspaces corresponding to $i$-valued $1$-form part and spinor part.  
%Here, $\left\{\hat{A}_{b}\right\}_{b \in B}$ denotes a family of $U(1)$-connections of the family of the determinant line bundles and $\left\{\ind{D^{+}_{\hat{A_{b}}}}\right\}_{b \in B}$ denotes the index of the families of the Dirac operators associated to $E$.
\end{itemize}

\end{thm}

The map ${\bf BF}_E$ is called the \emph{homotopy coherent Bauer--Furuta invariant}. We do not need to treat invariance problems on choices of a fiberweise spin$^c$ structures in this paper. Thus, we do not discuss it here although there seems not to be any issue except for a choice of lifting of family spin$^c$ structures.

\begin{rem}
    If the boundary family comes from the trivial action, the corresponding Bauer--Furuta invariant coincides with the family Bauer--Furuta invariant constructed in \cite{KT22}, which is a family version of \cite{BF04, Ma03, Kha15}.  
    
    %Also, we will not care about the relation between the Bauer--Furuta invariant on $(B\mathbb{Z}_p)_n$ and that on $(B\mathbb{Z}_p)^{n+1}$ although there should be some relation between them. We shall just use the existence of such a map for a sufficiently large $n\gg 0$.
\end{rem}

\subsection{Construction of Bauer--Furuta type invariant}

We start with several basic notations related to 4-manifold family and Seiberg--Witten theory.  

\subsubsection{Notations of spaces of sections}
We construct the map ${\bf BF}_{E}$. We fix the same geometric data as in \Cref{subsec:BF_boundary}, yielding a family
    $$(X, \mathfrak{s}) \to E \to B\Z_p,   $$
which induces a family of spin$^c$ 3-manifolds 
\[
(Y, \mathfrak{t}) \to E_\partial \to B\Z_p
\]
that is isomorphic to a Borel construction $E\Z_p \times_{\Z_p} (Y, \mathfrak{t})$. Moreover, fix a fiberwise metric $g_E$ on $E \to B\Z_p$ such that, near a collar neighborhood $[-\varepsilon, 0 ] \times  (E_\partial)_b$ of $(E_\partial)_b$,
 \[
 g_E|_{[-\varepsilon, 0 ] \times  (E_\partial)_b}  = \pi^{\ast}_b g_{Y,b}+ dt^2,
 \]
 where $g_{Y,b}$ is a fiberwise metric $\partial (E_b) \to B\Z_p$ and $\pi_b \colon E_b \to \partial (E_b)$ is the projection. Moreover, the boundary Riemann metric is assumed to be obtained from the Borel construction with respect to the $\Z_p$-action.
 The existence of such a fiberwise metric is verified by taking a fiberwise metric on the Borel construction and extending it to $X$ fiberwise.

% This induces a family of spin$^c$ 3-manifolds 
% \[
% (Y, \mathfrak{t}) \to E_\partial \to BG
% \]
% which is isomorphic to a Borel construction
% \[
% EG \times_G (Y, \mathfrak{t}) 
% \]
% from the assumption. 

% as in the previous section we are fixing the the following input data: 

% \begin{itemize}
%     \item A compact oriented spin$^c$ 4-manifold $(X, \mathfrak{s})$ with a spin$^c$ rational homology 3-sphere boundary  $(Y, \mathfrak{t})$ such that 
%     \[
%     b_1(X) =0 \qquad\text{ and }\qquad b_2^+(X)=0 . 
%     \] 
%     \item A homotopy coherent $G$-spin$^c$ action ${\bf H_\mathfrak{s}}$ on $X$ which is strict on the boundary. 
% \end{itemize}

% \begin{itemize}
%     \item A compact oriented spin$^c$ 4-manifold $(X, \mathfrak{s})$ with a spin$^c$ rational homology 3-sphere boundary $(Y, \mathfrak{t})$ such that 
%     \[
%     b_1(X) =0 \qquad\text{ and }\qquad b_2^+(X)=0, 
%     \]
%     and moreover, assume that the $G$-action on $Y$ preserves the isomorphism class of $\mathfrak{t}$.
%     \item A map 
%     \[
%     BG \to B\mathrm{Diff}^+ (X ; [\fraks]) 
%     \]
%     defines a family 
%     \[
%   (X, \mathfrak{s}) \to E \to BG.   
%     \]
%     We assume that the induced family of $Y$ from $E$ is isomorphic to the Borel construction of a strict $G$-action.
% \end{itemize}

% From them, we can associate a fiber bundle 
% \[
% (X, \mathfrak{s}) \to E \to BG. 
% \]

 Let $\{\widehat{A}_b\}_{b \in B\Z_p}$ be a fiberwise reference spin$^c$-connection on $E$ such that $\widehat{A}_b|_{(E_\partial)_b}$ is a fiberwise spin$^c$ connection on $Y_b$ for any $b \in B\Z_p$ comes from a Borel construction of a $\Z_p$-invariant connection. 
Once we fix the data $(E, g_E)$, the following families of vector bundles over $B\Z_p$
\[
S^+_E,\quad  S^-_E,\quad i \Lambda^*_E,\quad i\Lambda^+_E, \quad S_{E_\partial} , \quad i \Lambda^*_{E_\partial}
\]
are associated.
The restrictions of them over $b \in BG$ are the positive and negative spinor bundles with respect to $(g_{E_b}, \fraks)$, and $i\Lambda^*_{X}$ and  $i\Lambda^+_{X}$ with respect to $g_{E_b}$ respectively, where $\Lambda^+_{X}$ denotes the space of self-dual part of $i\Lambda^2_X$. The induced bundles from $ E_\partial$ are similar, $S_{E_\partial}$ is a family spinor bundle, and $i \Lambda^*_{E_\partial}$ is the induced bundle whose fiber is $i \Lambda^*_Y$. 
We shall use the notation
\[
L^2_k (S^+_E),\quad  L^2_k(S^-_E),\quad  L^2_k (i\Lambda^*_E),\quad  L^2_k  (i \Lambda^+_E), \quad L^2_{k-\frac{1}{2}} ( S_{E_\partial}) , \quad L^2_{k-\frac{1}{2}} ( i \Lambda^*_{E_\partial})
\]
to denote the spaces of fiberwise $L^2_k$-sections and $L^2_{k-\frac{1}{2}}$-sections of the bundles above. 
In order to obtain the Fredholm property for a certain operator, we shall use a subspace $L^2_{k} ( i\Lambda^1_{E})_{CC}$ of $L^2_{k} ( i\Lambda^1_{E})$ defined by
\[
L^2_k (i\Lambda^1_E)_{CC} :=\bigsqcup_{b \in BG}  \Set{ a \in L^2_k (i  \Lambda ^1_{E_b})   \mid d^* a =0,\  d^*{\bf t}a=0 }, 
\]
where ${\bf t}$ is the restriction as differential forms along the inclusion $Y= (E_\partial)_b  \inc E_b$. 
This gauge fixing condition is called the {\it double Coulomb condition} and was introduced by Khandhawit~\cite{Kha15}.
We define 
\[
{\bf V} (E_\partial) := \bigsqcup_{b \in BG}  \Set{ a \in L^2_{k-\frac{1}{2}} (i  \Lambda ^1_{(E_\partial)_b})   \mid d^* a =0 } \oplus L^2_{k-\frac{1}{2}} ( S_{E_\partial}) 
\]
which equips a natural fiber bundle structure on $B\Z_p$.

\subsubsection{Seiberg--Witten equation in our setting}
In general, taking a fiberwise finite-dimensional approximation of $V(E_\partial)$ is impossible if there is a non-trivial spectral obstruction to get a spectral section, see \cite[Section 2]{SS21}. So, we need to care about such an obstruction. That is why we are assuming that the boundary action is strict. 
Since we assume there is an isomorphism 
\[
E_{\partial } \cong E\Z_p \times_{\Z_p}  (Y, \frak{t}), 
\]
there is an isomorphism as Hilbert bundles: 
   \[
 \begin{CD}
      {\bf V}(E_\partial)   @>{\Phi}>>  V(Y)\times_{\Z_p}  E\Z_p \\
  @VVV    @VVV \\
     B\Z_p     @>{\mathrm{id}}>> B\Z_p  .  
  \end{CD}
\]

For the compactness of Seiberg--Witten moduli space, from now on, we shall restrict all families to its $n$-skeleton $(B\Z_p)_n$; its definition was given in \Cref{sec:homotopycoherentactions}. We denote its preimage under the covering map $E\Z_p \rightarrow B\Z_p$ by $(E\Z_p)_n$.

Via this identification $\Phi$, we take a fiberwise finite-dimensional approximation of $V(E_\partial)$. For a given real number $\mu$, we define 
\[
{\bf V}^\mu_{-\infty} (E_\partial) :=(E\Z_p)_n \times _{\Z_p}  V^\mu_{-\infty} (Y)  \subset (E\Z_p)\times_{\Z_p}  V(Y). 
\]
 This construction depends on $n$ but we abbreviate it for the simple notations. 

Now for a real number $\mu$, we have the fiberwise Seiberg--Witten map over a slice
\[
\mathcal{F}^\mu : L^2_{k} ( i\Lambda^1)_{CC}  \oplus L^2_k ( S^+_E)  \to     L^2_{k-1} ( i\Lambda^+)\oplus L^2_{k-1} ( S^-_E)  \oplus {\bf V}^\mu_{-\infty} (E_\partial)
\]
defined by 
\[
\mathcal{F}^\mu \left( (A_b, \Phi_b )_{b \in (B\mathbb{Z}_p)_n} \right) = \left(  \rho_b (F^+ (A_b)) -  (\Phi_b, \Phi_b)_0 , {D}_{\widehat{A}_b+ A_b} ( \phi), p^\mu_{-\infty } \Phi r_b ( A_b, \Phi_b ) \right)_{b \in (B\mathbb{Z}_p)_n} , 
\]
where $F^+ (A_b)$ is the self-dual part of the curvature of a fiberwise connection $A_b$, $\rho_b$ is the Clifford multiplication, ${D}_{\widehat{A}_b+ A_b}$ is the fiberwise Dirac operator with respect to a connection $\widehat{A}_b+ A_b$, and 
\[
r_b \colon L^2_{k} ( i\Lambda^1)_{CC}  \oplus L^2_k ( S^+_E) \to  {\bf V}(E_\partial)
\]
 is the fiberwise restriction. Here, $p^\mu_{-\infty}$
 is the projection. 
 
We decompose $\mathcal{F}^{\mu}$ as the sum of a fiberwise linear operator 
\begin{align*}
L^{\mu} = 
\left\{L^\mu_b= (d^+, D_{\widehat{A}_b },  p^\infty_{-\infty} r_b )\right\}_{b \in (B\mathbb{Z}_p)_n}
\end{align*}
 and a fiberwise quadratic part
 \[
c^{\mu}
= \left\{ c^\mu_b  = (- (\Phi_b \otimes \Phi^*_b)_0, \rho (A_b) (\Phi_b), 0) \right\}_{b \in (B\mathbb{Z}_p)_n}.
 \] 
 We often use a decomposition of the operator $L^\mu_b$ for each $b$ as the sum of the real operator 
 \[
 L^\mu_{b, \R}= \left(d^+, 0,  (p^\mu_{-\infty})_{\R}r_b \right) \colon L^2_{k} ( i\Lambda^1_{E_{b}})_{CC} \to     L^2_{k-1} ( i\Lambda^+_{E_{b}})\oplus V^\mu_{-\infty}(\R)_b
 \]
 and the complex operator 
 \[
 L^\mu_{b, \C} =\left(0 , D_{\widehat{A}_b },  (p^\mu_{-\infty})_{\C}r_b \right) \colon  L^2_k ( S^+_{E_{b}})  \to    L^2_{k-1} ( S^-_{E_{b}})  \oplus V^\mu_{-\infty}(\C)_b,
 \]
where $(p^\mu_{-\infty})_{\R}$ and $(p^\mu_{-\infty})_{\C}$ denote the real and complex parts of $p^\mu_{-\infty}$ respectively.
  
It is checked in \cite[Proposition 2]{Kha15} that the fiberwise linear operator ${L}^\mu_b$ is Fredholm on each fiber and the Fredholm index is given by 
\[
2\ind_{\mathbb{C}}^{APS} D^+_{A_b}  - \dim V^\mu_0 , 
\]
 where $\ind_{\mathbb{C}}^{APS} D^+_{A_b}$ is the Fredholm index of $L^0_{b, \C}$ as a complex operator with respect to Atiyah--Patodi--Singer's boundary condition.  

The following lemma provides fundamental properties of the linear map $L^0_{b, \R} \colon  L^2_{k} ( i\Lambda^1_{E_{b}})_{CC}  \to   L^2_{k-1} ( i\Lambda^+_{E_{b}})\oplus V^0_{-\infty} (\R)$. 
%Although the bundle $\mathcal{H}^+(E)$ depends on the choice of fiberwise Riemann metric, we see that its isomorphism class is independent of choices:

\begin{lem}
\label{linear injection}
Under the assumption that $b_1(X)=0$ and  $b^+(X)=0$, 
the operator
\[
L^0_{b, \R} \colon  L^2_{k} ( i\Lambda^1_{E_{b}})_{CC}  \to   L^2_{k-1} ( i\Lambda^+_{E_{b}})\oplus V^0_{-\infty} (\R)
\]
is an isomorphism for any $b \in (B\mathbb{Z}_p)_n$.
\end{lem}
\begin{proof}
In \cite[Proposition 2]{Kha15}, it is confirmed that there are natural isomorphisms: 
\[
\mathrm{Ker}\,  L^0_{b, \R} \cong H^1(X; \R) =\{0\} \qquad\text{ and }\qquad \mathrm{Coker}\, L^0_{b, \R}  \cong H^+ (X ;\R) =\{0\}
\]
on each point $b \in (B\mathbb{Z}_p)_n$.
Thus, we complete the proof. 
 \end{proof}

\subsubsection{Several analytical lemmas}

From now on, we follow the construction of a family version of Bauer--Furuta invariant basically proven in \cite{KT22}. Most of the parts of the construction are similar to the arguments given in \cite[Section 2.3]{KT22}.

In order to carry out finite-dimensional approximation, we take a sequence of finite-dimensional vector subbundles $W_1^m$ of  $L^2_{k-1} ( i\Lambda^+_{E})\oplus L^2_{k-1} ( S^-_E)$. 

\begin{lem} \label{fin dim app1} There exists a sequence of finite-dimensional vector subbundles $W_1^m$ over $(B\mathbb{Z}_p)_n$ in  $L^2_{k-1} ( i\Lambda^+_{E})\oplus L^2_{k-1} ( S^-_E)$
 such that 
 \begin{itemize}
 \item the sequence is an increasing sequence 
 \[
 W_1^0 \subset W_1^1 \subset W_1^2 \subset W_1^3 \subset \cdots \subset L^2_{k-1} ( i\Lambda^+_{E})\oplus L^2_{k-1} ( S^-_E), 
 \]
 \item the equality 
\begin{align}\label{trans}
\mathrm{Im}  \left( \mathrm{pr}_{ L^2_{k-1} ( i\Lambda^+_{E_b})\oplus L^2_{k-1} ( S^-_{E_b})} \circ L^0_{b}  \right) +   (W_1^0)_{b} = L^2_{k-1} ( i\Lambda^+_{E_b})\oplus L^2_{k-1} ( S^-_{E_b})
\end{align}
holds for all $b \in (B\mathbb{Z}_p)_n$, and
\item  the projection $\pr_{(W_1^m)_b } : L^2_{k-1} ( i\Lambda^+_{E_b})\oplus L^2_{k-1} ( S^-_{E_b}) \to (W_1^m)_b$ satisfies 
\[
\| \pr_{(W_1^m)_b } \gamma_b  - \gamma_b \|_{L^2_{k-1}} \to 0 \text{ as } n \to \infty
\]
 for any $\gamma_b \in  L^2_{k-1} ( i\Lambda^+_{E_b})\oplus L^2_{k-1} ( S^-_{E_b})$  and $b\in (B\mathbb{Z}_p)_n$.
\end{itemize}
\end{lem}
\begin{proof}
The proof is exactly the same as in \cite[Lemma 2.11]{KT22}. In the proof of \cite[Lemma 2.11]{KT22}, the triviality of the induced $3$-manifold family from $E$ is assumed but that proof does not rely on the triviality. 
\end{proof}

Take sequences of numbers $\lambda_m$ and $\mu_m$ such that $\lambda_m \to -\infty$ and $\mu_m \to \infty$ as $m \to \infty$, and take a sequence of finite-dimensional vector subbundles $W_1^m$ over $(B\mathbb{Z}_p)_n$ of $L^2_{k-1}(i\Lambda^+_E) \oplus L^2_{k-1}(S^-_E)$ satisfying the conclusions of \cref{fin dim app1}.
Let us define 
\[
W_0^m:= ( {L}^{\mu_m} )^{-1} ( W_1^m \oplus {{\bf V}_{\lambda_m}^{\mu_m}}  ) \to (B\mathbb{Z}_p)_n .
\]
 Then for all sufficiently large $m$, we see that $W_0^m \to (B\mathbb{Z}_p)_n$ are finite dimensional subbundles of $ L^2_{k} ( i\Lambda^1_E)_{CC}  \oplus L^2_k ( S^+_E)$. The following lemma tells us the injectivity of $L^\mu_b$ for a sufficiently large $\mu$: 
\begin{lem} \label{inj} There exists $\mu_0>0$ such that, for any $\mu$ with $\mu>\mu_0$ and for any $b \in (B\Z_p)_n$, 
 $L^\mu_b$ is injective. 
 \end{lem}

We have an isomorphism 
\[
W_1^m + {V^{\mu_m}_{\lambda_m}} + \mathrm{Ker}\, L^{\mu_m} \cong W_0^m + \mathrm{Coker}\, L^{\mu_m}
\]
between the virtual vector bundles.
As it is mentioned in \cite[page 923]{Ma03}, 
\[
\mathrm{Coker}\, L^{\mu_m} \cong \mathrm{Coker}\, L^{0}  \oplus  \mathrm{Coker}\, (p_{0}^{\mu_m} \circ  \pr_{\Ker d^{\ast}}). 
\]
From \cref{inj}, we know that, for a sufficiently large $m$, the operator $L^{\mu_m}_b= (d^+, D_{\widehat{A}_b },  p^\mu_{-\infty} r_b )$ is injective for every $b\in (B\mathbb{Z}_p)_n$. Thus 
\[
p_{0}^{\mu_m} \circ  \pr_{\Ker d^{\ast}}  \circ i^{\ast} : \Ker L^{0} \to {{\bf V}_{0}^{\mu_m}}
\]
 is fiberwise injective, and we have an identification
\[
\Ker L^0-  \mathrm{Coker}\, L^0 + \mathrm{Coker}\, L^{\mu_m} \cong {{\bf V}^{\mu_m}_0},
\]
and thus have
\begin{align}\label{decomp}
W_1^m + {\bf V}^{\mu_m}_{\lambda_m} + \Ker L^0- \mathrm{Coker}\, L^0 \cong   W^m_0 + {{\bf V}^{\mu_m}_0}
\end{align}
as virtual vector bundles over $(B\Z_p)_n$.

Applying the projection, we obtain a family of maps
\[
\pr_{W_1^m\times {\bf V}^{\mu_m}_{\lambda_m}}  \circ \mathcal{F}^{\mu_{m}} |_{W_0^m} :W_0^m \to W_1^m \times {{\bf V}_{\lambda_m}^{\mu_m}}
\]
whose $S^1$-invariant part is given by 
\[
(\pr_{W_1^m\times {\bf V}_{\lambda_m}^{\mu_m}}  \circ \mathcal{F}^{\mu_{m}} |_{W_0^m})^{S^1}  :W_0^m(\R)  \to W_1^m(\R) \times  {\bf V}^{\mu_m}_{\lambda_m}(\R).
\]
This induces a map 
\begin{align}\label{fin app}
\pr_{W_1^m\times {\bf V}_{\lambda_m}^{\mu_m} }  \circ \mathcal{F}^{\mu_{m}} |_{W_0^m} :B(R,W_0^m) \to (W_1^m \times  {\bf V}_{\lambda_m}^{\mu_m} )^{+_{(B\mathbb{Z}_p)_n}} , 
\end{align}
where $+_{B}$ denotes the fiberwise one-point compactification. 

We endow the vector bundle ${\bf V}_{\lambda_m}^{\mu_m}$ with a $\R$-action obtained as the flow equation 
\[
\frac{\partial}{\partial t} y (t) = - (l + p_{\lambda_m}^{\mu_m} c) (y(t)). 
\]
Note that, for a large $R\gg 0$, the disk bundle $B(R; {\bf V}^{\mu_m}_{\lambda_m} )$ is an isolating neighborhood of this dynamical system. 
The following lemma gives a sufficient condition to get a finite-dimensional approximation of the Seiberg-Witten map in our situation. 
To obtain a suitable index pair used for a Bauer--Furuta-type invariant from \eqref{fin app}, we need a certain compactness type result, described as follows.

\begin{lem}
\label{lem: ensure the existence of Conley index}
The following statements hold. 
\begin{itemize}
    \item
    Set
\begin{align*}
\wt{K}_{1 }  := B(R,W_0^m) \cap ( (\pr_{W_1^m}  \circ \mathcal{F}^{\mu_m})^{-1}    B(\epsilon_m, W_1^m) ) 
\end{align*}
and 
\begin{align*}
\wt{K}_{2} := S(R,W_0^m) \cap ( (\pr_{W_1^m }  \circ \mathcal{F}^{\mu_m})^{-1}    B(\epsilon_m, W_1^m) )  
\end{align*}
for a sequence of positive real numbers $\{ \epsilon_m \}_{ m \in \Z_{>0}} $ with $ \epsilon_m \to 0$ as $m\to \infty$. 
For sufficiently large $R, R'$ and $m$, the compact sets
\[
K_{1 }  := p_{{\bf V}_{\lambda_m}^{\mu_m} }  \circ \mathcal{F}^{\mu_m}(\wt{K}_{1})
\qquad \text{ and } \qquad
K_{2} :=  p_{{\bf V}_{\lambda_m}^{\mu_m} }  \circ \mathcal{F}^{\mu_m}(\wt{K}_{2})
\]
satisfy the assumption of \cite[Theorem~4]{Ma03}, \cite[Lemma~A.4]{Kha15} for $A:=B(R' ; {\bf V}^{\mu_m}_{\lambda_m} )$, i.e.\ the following conditions hold: 
\begin{itemize}
\item[(i)]  if $x \in K_{1} \cap A^+ $, then $([0,\infty) \cdot x) \cap \partial A=\emptyset$, and 
\item[(ii)] $K_{2} \cap A^+  = \emptyset $, where 
$$A^{+} := \Set{x \in A \mid \forall t>0,\  t \cdot x \in A}$$
for a subset $A$ in ${\bf V}_{\lambda_m}^{\mu_m}$. 
\end{itemize}

% \begin{align*}
% &A^{+} := \Set{x \in A \mid \forall t>0,\  t \cdot x \in A}, 
% \end{align*}
% for a subset $A$ in $V_{\lambda}^{\mu}$. 

\item  Let $\{ x_m\}_{  n\in \Z_{>0}} $ be a bounded sequence in $L^2_{k} ( i\Lambda^1_E)_{CC}  \oplus L^2_k (S^+_{E} ) $ such that 
\[
L^{\mu_m}_{b_m} x_m \in    W_1^m ,\qquad  p^{\mu_m}_{-\infty} r_{b_m} x_m\in   {\bf V}_{\lambda_m}^{\mu_m},  \qquad \text{ and } \qquad (L^{\mu_m}_{b_m}  + p^{\mu_m}_{-\infty}  C_{b_m}  ) (x_m) \to 0 
\]
 in $L^2_k$-norm, where $b_m$ is the corresponding base point of $x_m$ in $(B\Z_p)_n$.
  We also suppose that there exists a sequence of approximated half trajectories $y_m : [0,\infty) \to {\bf V}_{\lambda_m}^{\mu_m} $ satisfying
\[
\frac{\partial}{\partial t} y_m (t) = - (l + p_{\lambda_m}^{\mu_m} c) (y(t))  \qquad \text{ and }\qquad y_m(0) =p^{\mu_m}_{-\infty} r_{b_m}  x_m. 
\]
Then, after taking a subsequence, the sequence $\{ b_m \}$ converges a point $b_\infty \in (B\Z_p)_n$, the sequence $\{x_m\}$ converges to a solution $x_\infty$ to the Seiberg--Witten equations for $E_{b_\infty}$ and there exists a Seiberg--Witten half trajectory $y_\infty$ satisfying $\frac{\partial}{\partial t} y (t) = - (l +  c) (y(t)) $,  $y_\infty(0)=r_{b_\infty}  x_\infty$ and $y_m(t)$ converges $y_\infty(t)$ for all $t$ in $L^2_{k- \frac{1}{2}}$. 
\end{itemize}

\end{lem}

\begin{proof}
Again, we did not use the triviality of the $3$-manifold bundle to ensure compactness in~\cite[Lemma 2.13 and Lemma 2.14]{KT22}; thus, that proof does work in our situation as well.
\end{proof}

\subsubsection{Construction of ${\bf BF}_E$}
By using the above lemma and \cite[Theorem~4]{Ma03} to the $S^1$-equivariant dynamical system ${\bf V}^{\mu_m}_{\lambda_m} (E_\partial)$, we may take an $S^1$-invariant Conley index $(N_m, L_m)$ of ${\bf V}^{\mu_m}_{\lambda_m} (E_\partial)$ such that 
\[
(K_{1 }, K_2 ) \subset (N_m, L_m) , 
\]
 although Manolescu assumes a manifold structure to take such a Conley index in \cite[Theorem~4]{Ma03}, his argument works for ${\bf V}^{\mu}_\lambda (E_\partial)$, which is a manifold on each fiber and $\R$-action preserves the fiber. 

Since $\R$-action preserves the fiber of $\pi: {\bf V}^{\mu_m}_{\lambda_m} (E_\partial) \to (B\Z_p)_n$, the restriction $(N_m \cap \pi^{-1} (b) , L_m\cap \pi^{-1} (b) )$ to each fiber is also an index pair.

\begin{lem}\label{choice of CI}
   One can see 
   $
   \bigcup_{b\in (B\Z_p)_n} N_m \cap \pi^{-1} (b) /  L_m\cap \pi^{-1} (b) 
   $ is fiberwisely and $S^1$-equivariantly homotopy equvalent to
   \[
   (E\Z_p)_n \times_{\Z_p } (N^*_m/  L_m^*), 
   \]
where a pair $(N^*_m,L_m^*)$ is an $S^1\times \Z_p$-equivariant index pair of the dynamical system on a fiber ${ V}^{\mu_m}_{\lambda_m} (Y)$.
\end{lem}
\begin{proof}
    This is a fiberwise version of the independence of choices of Conley indices. A homotopy equivalence between two Conley indices is given by using the $ \R$-action. 
Since $(N_m, L_m)$ and $((E\Z_p)_n \times_{\Z_p } N^*_m,   (E\Z_p)_n \times_{\Z_p } L_m^*)$ are Conley indices of the dynamical system on ${\bf V}^{\mu_m}_{\lambda_m} (E_\partial)$, we have a homotopy equivalence 
\[
f: N_m/L_m \to (E\Z_p)_n \times_{\Z_p } N^*_m/  (E\Z_p)_n \times_{\Z_p } L_m^*
\]
which is given by the $\R$-action with shifted time parameters. This map has a representative $f: N_m \to (E\Z_p)_n \times_{\Z_p } N^*_m$ satisfying 
\[
f ( L_m ) \subset (E\Z_p)_n \times_{\Z_p } L_m^*. 
\]
For the definition of the maps, $f$, and a homotopy inverse, see \cite[Definition 6.1, Theorem 6.3, Theorem 8.1]{RS88}. Such invariance is also observed in \cite{Co78} and \cite[Section 6]{Sa85}. 
Since $\R$-action preserves each fiber, one can see that this flow map lifts to 
\[
f': \bigcup_{b\in (B\Z_p)_n} N_m \cap \pi^{-1} (b) /  L_m\cap \pi^{-1} (b) \to (E\Z_p)_n \times_{\Z_p } N^*_m/  (E\Z_p)_n \times_{\Z_p } L_m^*
\]
Note that this map $f'$ is an $S^1$-equivariant fiberwise map since $\R$-action is $S^1$-equivariant and preseves each fiber. Moreover, a homotopy inverse $g$ and homotopies between $f'\circ g$ and $\id$ ( and $g \circ f'$ and $\id$ ) are given by $\R$-actions with shifted time parameters, they are also $S^1$-equivariant fiberwise maps. This completes the proof.  
 \end{proof}
Then $\pr_{W_1^{m}}  \circ  \mathcal{F}^{\mu_{m}}|_{W_{0}^{m}}$ combined with the $S^1$ equivariant fiberwise homotopy obtained in \cref{choice of CI} induces an $S^1$-equivariant continuous map
\begin{align} \label{simplify}
{\bf BF}_E :  (W_0^m(R))^{+_{(B\mathbb{Z}_p)_n}}   \to  (W_1^m / (W_1^m \smallsetminus B(\epsilon_m, W_1^m ))) \wedge_{(B\mathbb{Z}_p)_n} (E\Z_p)_n \times _{\Z_p} {I_{\lambda_m}^{\mu_m} }
 \end{align}
 as in \cite[Section~9]{Ma03},
 where $\wedge_{(B\mathbb{Z}_p)_n}$ denotes the fiberwise smash product. 
We call this map \eqref{simplify} the {\it families relative Bauer--Furuta invariant}.

 The decomposition \eqref{decomp} implies that this map stably can be written so that 
 \begin{align*}
 \begin{split}
{\bf BF}_E :  \left(\Set{\ind D^{+}_{\widehat{A}_{b}}}_{b \in (B\Z_p)_n}\right)^{+_{(B\mathbb{Z}_p)_n}}
 \to (E\Z_p)_n \times _{\mathbb{Z}_p} {SWF(Y, \frakt, g_Y)},
  \end{split}
 \end{align*}
 where $\left\{\ind D^{+}_{\widehat{A}_{b}}\right\}_{b \in (B\Z_p)_n}$ denotes the virtual index bundle derived from the family Dirac index with Atiyah--Patodi--Singer's boundary condition. Here we consider stabilizations by real vector bundles with the trivial $S^1$ actions or complex vector bundles with the complex multiplications and used $b_2^+(X)=0$.

\begin{proof}[Proof of \cref{hc Bf inv}]
We check the conditions in \cref{hc Bf inv} of ${\bf BF}_{E}$: 
\begin{itemize}
    \item[(i)] \eqref{simplify} implies the desired property. 
    \item[(ii)] This follows from the choices of finite dimensional approximations $W_0$ and $W_1$. 
    \item[(iii)] The property of ${\bf BF}_{E}^{S^1}$ follows from \cref{linear injection}. 
    \item[(iv)] The computation of the Dirac index can be checked fiberwisely. So, it follows from \cite[Proposition 2]{Kha15}. 
\end{itemize}
This completes the proof. 
\end{proof}
In our situation, we will use a negative-definite 4-manifold cobordism from $-Y$ to the empty set and a certain family of it. We shall consider $Y = -\Sigma(a_1, \dots, a_n)$ in this notation for the proof of the main theorem. Therefore, we need to regard the map \eqref{simplify} as a map from $(E\Z_p)_n \times SWF(-Y, \mathfrak{t})$ to a fiberwise Thom space of a bundle. So, we shall take the fiberwise wedge product with $SWF(-Y, \mathfrak{t}, g)$ and compose the Borel construction of the equivariant duality map 
\[
\epsilon : SWF(Y, \mathfrak{t}, g) \wedge SWF(-Y, \mathfrak{t}, g)  \to S^{-k(D)\C}, 
\]
where $k(D)$ means the dimension of the kernel of the Dirac operator and obtain 
\begin{align}\label{dagger version}
    \left(\Set{\ind D^{+}_{\widehat{A}_{b}}}_{b \in B}\right)^{+_{(B\mathbb{Z}_p)_n}} \wedge_{(B\mathbb{Z}_p)_n} (E\Z_p)_n \times _{\Z_p} {SWF(-Y, \frakt, g_Y)}
 \to {\bf S^0} 
\end{align}
which we will write by ${\bf BF}^\dagger_{E}$, where ${\bf S^0} $ denotes the trivial $S^0$ bundle over $(B\mathbb{Z}_p)_n$.
Here we consider the up-side-down cobordism from $-Y$ to the empty set to define the family of Dirac indices. 
For the construction of $\epsilon$, see \cite{Baraglia-Hekmati:2024-1}. 
Since $\epsilon^{S^1}$ is degree one map,  ${\bf BF}^\dagger_{E}$ is a local map in the sense of \cref{def:local}. 

%{\color{red} Add some more explanations to it. }

\section{Local maps from homotopy coherent finite group action }

\subsection{Review of Baraglia--Hekmati's theory}
Since our main theorem gives an inequality for the equivariant Fr\o yshov invariant, we review its construction due to Baraglia--Hekmati~\cite{Baraglia-Hekmati:2024-1}.  As before, assume that $Y$ is a rational homology $3$-sphere with a spin$^c$ structure $\mathfrak{t}$, and that there is a smooth $\Z_p$-action preserving both the orientation and the isomorphism class of $\mathfrak{t}$, where $p$ is prime. Lastly, let $g_Y$ be a $\Z_p$-invariant Rieman metric on $Y$.

As in \Cref{Bauer--Furuta invariant for homotopy coherent finite group actions} we have
\[
SWF(Y, \mathfrak{t}, g_Y)=\Sigma^{-V^0_\lambda(g_Y)}I^\mu_\lambda(g_Y). 
\]
for sufficiently large real numbers $\mu, -\lambda$.
There is a version of an equivariant Floer homology in  \cite{Baraglia-Hekmati:2024-1}:  
\[
H_{*}^{S^1\times \Z_p} (SWF(Y, \mathfrak{t})) := H_{*+2n(Y, \mathfrak{t} , g_Y)}^{S^1\times \Z_p} (SWF(Y, \mathfrak{t}, g_Y)) = H_{*+2n(Y,\mathfrak{t}, g_Y)}^{G_\mathfrak{t}} (SWF(Y, \mathfrak{t}, g_Y)). 
\]
 Then 
 $H_{*}^{S^1\times \Z_p} (SWF(Y, \mathfrak{t}))$
 is a module over the ring $H^*_{S^1\times \Z_p} := H^*_{S^1\times \Z_p}(pt; \Z_p)$.
Let us recall that 
\[
H^*_{S^1\times \Z_p} =
\Z_p[U, R, S] / (R^2)
\]
where $\deg(U) = 2$, $\deg(R) = 1$, and $\deg(S) = 2$.

\subsection{Equivariant Fr\o yshov invariant $\delta^{(p)}_i (Y)$}
We summarize definitions and properties of equivariant Fr\o yshov invariants $\delta^{(p)}_i (Y)$, which are introduced in \cite{Baraglia-Hekmati:2024-1}.

Since the equivariant Seiberg--Witten Floer homotopy type has advantageous properties from the perspective of equivariant cohomology, we summarize these properties and refer to spaces that satisfy them as $\Z_p$-SWF type spaces.
\begin{defn}[{\cite[Definition 3.6]{Baraglia-Hekmati:2024-1}}] \label{def:BH}
A space $X$ is called a \emph{space of type $\Z_p$-SWF} if $X$ is a pointed finite $(S^1 \times \Z_p)$-CW complex, such that the following holds:
\begin{itemize}
    \item[(i)] The fixed point set $X^{S^1}$ is $\Z_p$-homotopy equivalent to a sphere $V^+$, where $V$ is a real representation of $\Z_p$.
    \item[(ii)]  The action of $S^1$ is free on $X \smallsetminus X^{S^1}$. 
\end{itemize}
\noindent We call the dimension of $V$ the \emph{level} of $X$.
\end{defn}

It is confirmed that the Seiberg--Witten Floer homotopy type $SWF(Y, \mathfrak{t}, g_Y)$ (with a certain stabilization by $S^1\times \Z_p$ vector spaces) is a space of type $\Z_p$-SWF for a given rational homology 3-sphere $Y$ with a $\Z_p$-action, and with an invariant spin$^c$ structure and Riemannian metric $g_Y$ under this action. 

The inclusion of the fixed points $\iota: X^{S^1} \to X$ induces a map  
\[
\iota^*: U^{-1} \wt{H}^*_{\Z_p} (X;\mathbb{F}_p) \to U^{-1} \wt{H}^*_{\Z_p} (X^{S^1} )\cong U^{-1}  H^{\ast+\text{level of }X}_{S^1\times \Z_p}(\ast;\mathbb{F}_p).
\]
The localization theorem implies $U^{-1} \wt{H}^*_{\Z_p} (X;\mathbb{F}_p)$ is a free rank-$1$ $U^{-1}  H^*_{S^1\times \Z_p} $-module. 
 Now we recall a sequence of invariants $ \{d_{j} (X)\}$.
Suppose $p$ is an odd prime. \footnote{For $p=2$, there is a similar formula, but we focus on odd prime cases in this paper. }
The equivariant Fr\o yshov type invariant is defined by 
\[
d_{j} (X) := \min \{ i \mid \exists x \in \wt{H}^i_{\Z_p}(X;\mathbb{F}_p) ,\ \iota^*  x \equiv  S^j U^k \; \mathrm{ mod } \; \langle S^{j+1}, RS^{j+1} \rangle  \text{  for some }  k \geq 0 \}-2j .
\] 

% Let $\tau$ denote a generator of $U^{-1}  H^*_{S^1\times \Z_p} $.

Now, we give definitions of the invariants. 
\begin{defn}
Let $\Z_p$ act on $Y$, and let $\frakt$ be a $\Z_p$-invariant spin$^c$ structure. The invariant $d_{j}(Y, \frakt)$ is defined as
\[
d_{j} (Y, \frakt) := d_{j} (SWF(Y, \frakt, g_Y) ) - 2n(Y, \frakt, g_Y), 
\]
and the invariant $\delta^{(p)}_j(Y, \frakt)$ is is defined as
\[
\delta^{(p)}_j (Y, \frakt) := \frac{1}{2}d_{j} (Y, \frakt).
\]
\end{defn}
\noindent It is proven that $\delta^{(p)}_{j} (Y, \frakt)$ becomes eventually constant as $j$ increases~\cite[Theorem 5.2]{Baraglia-Hekmati:2024-1}.
%This is not the original definition but an alternative description given in \cite[Proposition 3.14]{BH}. 

\subsection{Family Fr\o yshov invariants }
It is convenient to formulate a family version of Fr\o yshov invariants for spaces including the Borel constructions of Seiberg--Witten homotopy types.

\begin{defn} Let $B$ be a compact connected Hausdorff space. A \emph{space of family $SWF$ type over $B$} is a finite $S^1$-CW complex $X$ with an $S^1$-invariant projection 
\[
\pi\colon X \to B
\]
along with a section $B \to X$ whose image is $S^1$-invariant. The restriction
\[
\pi^{S^1}\colon  X^{S^1} \to B
\]
is a sphere bundle obtained from the fiberwise compactification of a vector bundle over $B$ and $X \smallsetminus X^{S^1}$ has a free $S^1$-action. We call the dimension of the fiber $\pi^{S^1}$ the \emph{level} of $(X, \pi, s)$.
\end{defn}

If we are given a space of type $\Z_p$-SWF $X$, we define 
\[
X' := E\Z_p \times_{\Z_p} X 
\]
which is naturally equipped with a projection $\pi \colon X' \to B\Z_p$ and a section $s \colon B\Z_p \to X'$. Note that ${H}^i_{S^1}(X, s(B)) $ has the action of $H^*(B)$ from the projection $\pi$.

\begin{rem}
    In our setting, we see $\pi$ is actually a fibration whose fiber is $S^1$-pointed space and the section $s$ comes from the base points.  
\end{rem}

Suppose $B=(B \Z_p) ^n$. We redefine Fr\o yshov invariant for a space of family $SWF$ in this setting. 
Note that the sphere bundle 
\[
\pi^{S^1} \colon X^{S^1} \to B
\]
must come from an oriented vector bundle over $B$ since $p$ is odd.

\begin{defn}
Fix an odd prime number $p$ and let $(X, \pi, s)$ be a space of $SWF$ type over $B = (B\mathbb{Z}_p)_n$. We then define the \emph{family Fr\o yshov invariant} by
    \[
    d_{j,n } (X) := \min \{ i \mid \exists x \in {H}^i_{S^1}(X, s(B)) ,\ \iota^*  x \equiv  S^jU^k    \;\mathrm{mod} \; \langle S^{j+1}, RS^j \rangle  \text{  for some }  k \geq 0 \}  -2j, 
    \] 
    where 
    $\iota \colon (X^{S^1},  s(B)) \to (X ,  s(B))$ 
    and 
    \[
    \iota^* \colon H^*_{S^1}(X,  s(B)) \to H^*_{S^1} (X^{S^1},  s(B)) \cong \F [U, R, S]/(R^2) [Th_{X^{S^1}}],  
    \]
    where $[Th_{X^{S^1}}]$ denotes the Thom class corresponding to a real oriented vector bundle. 
Here we have used the Thom isomorphism theorem for the Thom space of an oriented bundle 
    \[
    X^{S^1}/s(B) \to B. 
    \]
\end{defn}

\begin{defn}\label{def:local} Let $(X,  \pi, s)$ and $(X',  \pi', s')$ be spaces of family $SWF$  over $B$. A \emph{local map} between them is a $S^1$-continuous map $f\colon X \to X'$ satisfying 
\begin{itemize}
    \item $\pi' \circ f = \pi$, 
    \item $f \circ s = s'$, and
    \item $f^{S^1}$ is fiberwise homotopy equivalence.  
\end{itemize}
    
\end{defn}

While we will mainly deal with spaces of family SWF, we also need a setting where its assumptions are weakened.

\begin{defn}
    A \emph{space of family pre-SWF over a base $B$} is a topological space $X$, a continuous map $p_X\colon X\rightarrow B$, and its section $s_X\colon B\rightarrow X$.
\end{defn}

\begin{defn}
   Given two spaces $X=(X,p_X,s_X)$ and $Y=(Y,p_Y,s_Y)$ of family pre-SWF over $B$, we consider the fiber product $Z=X\times_B Y$, which is the subspace of $X\times Y$ consisting of pairs $(x,y)$ satisfying $p_X(x)=p_Y(y)$; it admits a canonical map $p\colon Z\rightarrow B$ induced by $p_X$ and $p_Y$. The \emph{fiberwise wedge product} $X\wedge_B Y$ is then defined as the quotient of $Z$ by the following relations:
    \begin{itemize}
        \item $(x,y)\sim (x^\prime,y)$ whenever $p_X(x)=p_X(x^\prime)=p_Y(y)$ and $y$ is contained in the image of $s_Y$;
        \item $(x,y)\sim (x,y^\prime)$ whenever $p_X(x)=p_Y(y)=p_Y(y^\prime)$ and $x$ is contained in the image of $s_X$.
    \end{itemize}
    This space is then endowed with the structure of a family pre-SWF over $B$ with the projection map $p\colon X\wedge_B Y\rightarrow B$ induced by the map $p\colon X\times_B Y\rightarrow B$, and its section $r\colon B\rightarrow X\wedge_B Y$ defined by $s(b)=(s_X(b),s_Y(b))$.
\end{defn}
\begin{rem}
    If $X,Y$ are sphere bundles (of fiber dimension $n$ and $m$, respectively) over $B$, then $X\wedge_B Y$ is also a sphere bundle (of fiber dimension $n+m$) over $B$. 
\end{rem}

\begin{rem}
    If $X$ is a space of family pre-SWF over $B$ and $S\rightarrow B$ is a trivial $S^1$-bundle over $B$, regarded as a sapce of family pre-SWF, their fiber wedge product $X\wedge_B S$ is the `fiber (unreduced) suspension' of $X$ over $B$, i.e.\ the quotient space of $X\times I$ by the relation $(x,t)\sim (y,t)$ if $t=0,1$ and $p(x)=p(y)$.
\end{rem}

\begin{rem}
If $X,Y$ are spaces of family SWF over B, so that they are now endowed with compatible $S^1$-actions, the fiber wedge product $X\wedge_B Y$ admits a canonical $S^1$-action compatible with $p\colon X\wedge_B Y\rightarrow B$ and $r\colon B\rightarrow X\wedge_B Y$. Furthermore, the $S^1$-fixed point locus of $X\wedge_B Y$ is given by 
\[
(X \wedge_B Y)^{S^1} = X^{S^1} \wedge_B Y^{S^1},
\]
which is a sphere bundle. Hence $X\wedge_B Y$ is a space of family SWF over $B$.
\end{rem}

\begin{defn}
    Given a space $X\rightarrow B$ of family SWF over $B$, we consider its $S^1$-Borel construction, which induces a map from $B(X)=(X\times \mathbb{S}^\infty)/S^1$ to $B\times \mathbb{CP}^\infty$. The given $S^1$-invariant section $B\rightarrow X$ induces a section $B\times \mathbb{CP}^\infty \rightarrow B(X)$. So we can consider $B(X)$ as a space of family pre-SWF over $B\times \mathbb{CP}^\infty$.
\end{defn}

\begin{rem}
\label{rem:product_borel}
    Given two spaces $X,Y\rightarrow B$ of family SWF, we have
    \[
    B(X)\wedge_{B\times \mathbb{CP}^\infty} B(Y) \cong B(X\wedge_B Y)
    \]
    as spaces of family pre-SWF over $B\times \mathbb{CP}^\infty$.
\end{rem}

Given an $S^1$-equivariant orientable rank $n$ vector bundle $E\rightarrow B$, the induced bundle $B(E^+)\rightarrow B\times \mathbb{CP}^\infty$ is an orientable $n$-sphere bundle. Also, for any space $X\rightarrow B$ of family SWF, we have a canonical identification
\[
H^\ast(B(X);\mathbb{F}_p)\cong H^\ast_{S^1}(X;\mathbb{F}_p),
\]
and thus the map
\[
H^\ast(B\times \mathbb{CP}^\infty;\mathbb{F}_p)\rightarrow H^\ast(B(X);\mathbb{F}_p)
\]
induces the $H^\ast_{S^1}(B;\mathbb{F}_p)$-module structure of $H^\ast_{S^1}(X;\mathbb{F}_p)$ given by the $S^1$-invariant map $X\rightarrow B$.

\begin{thm}
\label{thm:pre-SWF_thom}
    Let $p_X\colon X\rightarrow B$ (with section $s_X$) be a space of family pre-SWF over a base $B$, where $B$ is a CW complex with finitely many cells in each dimension. Given any orientable sphere bundle $p_S\colon S\rightarrow B$ with a section $s_S\colon B\rightarrow S$, we consider the fiber wedge product $X\wedge_B S$, which is again a space of family pre-SWF. Let $p_\wedge\colon X\wedge_B S\rightarrow B$ be the map induced by $p_X$ and $s_\wedge$ be its section induced by $s_X$ and $s_S$. Furthermore, suppose that $p_X$ is a fiber bundle whose fiber is a finite CW complex. Then we have 
    \[
    H^\ast(X,s_X(B);\mathbb{F}_p)\cong H^{\ast+\dim(S)}(X\wedge_B S,s_\wedge(B);\mathbb{F}_p)
    \]
    as $H^\ast(B;\mathbb{F}_p)$-modules.
\end{thm}
\begin{proof}
    Consider the canonical projection map $p\colon X\times_B S\rightarrow X$. Since $S$ is a sphere bundle over $B$, it follows that $p$ gives $X\times_B S$ a structure of a sphere bundle over $X$. Also, the section $s_S$ of $S$ induces a section $s_p$ of $p$. Thus we can apply relative Thom isomorphism theorem for the space pair $(X,s_X(B))$ to get an isomorphism
    \[
    H^\ast(X,s_X(B);\mathbb{F}_p) \cong H^{\ast+\dim(S)} (X\times_B S,s_p(X) \cup p^{-1}(s_X(B));\mathbb{F}_p),
    \]
    which is clearly $H^\ast(B;\mathbb{F}_p)$-linear.

    Now observe that the quotient map $X\times_B S\rightarrow X\wedge_B S$ identifies points $(x,s) \sim (x^\prime,s)$ whenever $p_X(x)=p_X(x^\prime)=p_S(y)$ and $(x,s)\sim (x,s^\prime)$ whenever $p_X(x)=p_S(s)=p_S(s^\prime)$. The subset of those points in $X\times_B S$ is entirely contained in $s_p(X) \cup p^{-1}(s_X(B))$, and its image under the quotient map is $s_\wedge(B)$. Therefore this quotient map induces an isomorphism \footnote{Technically, we need an additional step, as pairs of spaces involved are not assumed to be good pairs. But since we assumed $B$ to have finitely many cells in each dimension, we can simply prove that this map is an isomorphism in each degree by replacing $B$ with its very high dimensional skeleton.}
    \[
    H^\ast(X\wedge_B S,s_\wedge(B);\mathbb{F}_p) \xrightarrow{\cong} H^\ast(X\times_B S,s_p(X) \cup p^{-1}(s_X(B));\mathbb{F}_p),
    \]
    and thus the theorem follows.
\end{proof}

\begin{lem}
\label{lem:suspension_thom}
Let $E$ be a $S^1$-equivariant vector bundle over $B = (B\Z_p)_n$, which is the sum of a real vector bundle $E_\R$ with the trivial $S^1$-action and a complex vector bundle $E_\C$ with $S^1$ action as the complex vector bundle.  
For a space $(X, \pi, s)$ of family SWF over $B$, such that $\pi$ is a fiber bundle whose fiber is a finite CW complex, we have 
\[
d_{j , n} ( \Sigma^E X ) =  \dim_\R E_\C + d_{j,n}(X)
\]
for any $j$ and $n$. 
\end{lem}
%Now we can prove \Cref{lem:suspension_thom}.
\begin{proof}
    It suffices to prove the existence of an $H^\ast_{S^1}(B;\mathbb{F}_p)$-linear isomorphism
    \[
    H^\ast_{S^1}(X,s(B);\mathbb{F}_p)\cong H^{\ast+\dim(E)}_{S^1}(\Sigma^E X,s(B);\mathbb{F}_p).
    \]
    To see this, we note that $\Sigma^E X = X \wedge_B E^+$, and $E^+$ is an $S^1$-equivariant orientable sphere bundle over $B$. We then consider the induced spaces, $B(X)$ and $B(X\wedge_B E^+)$, which are spaces of family pre-SWF over $B\times \mathbb{CP}^\infty$. Since $E^+$ is orientable, $B(E^+)$ is also an orientable sphere bundle over $B\times \mathbb{CP}^\infty$. Hence we know from \Cref{rem:product_borel} and \Cref{thm:pre-SWF_thom} that we have an $H^\ast(B\times \mathbb{CP}^\infty;\mathbb{F}_p)$-linear isomorphism
    \[
    H^\ast(B(X),s(B);\mathbb{F}_p)\cong H^{\ast+\dim(E)}(B(X\wedge_B E^+),s(B);\mathbb{F}_p).
    \]
   This gives the desired equality. % Since the fiber dimension of $(\Sigma^E X)^{S^1}\rightarrow B$ is the fiber dimension of $X^{S^1}\rightarrow B$ plus the dimension of $E_\mathbb{R}$, the result follows.
\end{proof}

The following is immediate from the definition. 
\begin{lem}
    Let $X,Y\rightarrow B$ be a space of family SWF of level $m$ and $n$, respectively. Suppose that there exists a local map $X\rightarrow Y$, then $m=n$. \qed
\end{lem}

\begin{lem}
\label{lem:ineq}
If there is a local map $f\colon X \to X'$ between spaces of family $SWF$, $(X, \pi, s)$ and $(X', \pi', s')$, over $(B\mathbb{Z}_p)_n$, then for each $j$ we have
\[
d_{j,n} (X) \leq d_{j,n} (X').
\]
\end{lem}
\begin{proof}
    Since $X$ and $X^\prime$ have the same level, the existence of an $H^\ast(B; \mathbb{F}_p)$-linear map from $H^\ast(X', s^\prime(B); \mathbb{F}_p)$ to $H^\ast(X, s(B); \mathbb{F}_p)$ implies the given inequality.
\end{proof}

\subsection{Comparison between the equivariant and the family Fr\o yshov invariants}
Now we consider the following setting. Let $X$ be a finite $(S^1 \times \Z_p)$-CW complex satisfying the following conditions.
\begin{itemize}
    \item The $S^1$-fixed point locus $X^{S^1}$ of $X$ is homeomorphic to a sphere.
    \item There exists a preferred basepoint $x\in X$ which is invariant under the $S^1 \times \Z_p$ action.
\end{itemize}
Then we get an $X$-bundle $\wt{X}$ over $B\Z_p$, via Borel construction; it also carries a fiberwise $S^1$-action, whose fixed point locus forms a sphere bundle over $B\Z_p$. It also admits a preferred $S^1$-invariant section $s:B\Z_p \rightarrow \wt{X}$ induced by the $S^1 \times \Z_p$-invariant basepoint $x$. We further assume that the pullback map
\[
U^{-1}H^\ast_{S^1 \times \Z_p}(X,\{x\})\rightarrow U^{-1}H^\ast_{S^1 \times \Z_p}(X^{S^1},\{x\}),
\]
induced by the inclusion $(X^{S^1},\{x\})\subset (X,\{x\})$, is an isomorphism; note that $H^\ast_{S^1 \times \Z_p}(X^{S^1},\{x\})$ is free of rank 1 over the ring $H^\ast(B\Z_p)$ by Thom isomorphism theorem. Then its equivariant Fr\o yshov invariants $d_{j}(X)$ (for $j\ge 0$), as in \cite{Baraglia-Hekmati:2022-1}, are well-defined, as discussed after \Cref{def:BH}. We will denote it as $d_j(\wt{X})$, as it depends only on the family $\wt{X}$ induced from the Borel construction for the $\Z_p$-action on  $X$.

For any integer $n\ge 0$, we denote by $\wt{X}_n$ the restriction of the bundle $\wt{X}$ to the subspace $(B\Z_p)_n\subset B\Z_p$. Then it admits a canonical structure of a space of family SWF over the base $(B\Z_p)_n$. Then its family Fr\o yshov invariants are well-defined; we denote them as $d_{j,n}(\wt{X}_n)$.
%{\color{red} MT: I changed the notation $\delta$ to $d $ because of compatibility}
\begin{prop} \label{prop: stability}
    %Suppose the induced action of $\mathbb{Z}_p$ on $H^\ast_{S^1}(X,\{x\})$ is trivial. 
    For each $j \geq 0$, there exists a constant $C_j$ such that for any $n \geq C_j$, we have $d_{j,n}(\wt{X}_n) = d_j({X})$.
\end{prop}
\begin{proof}
    Observe that $B\Z_p$ is formed from $(B\Z_p)_n$ by attaching cells of dimension at least $n+1$. Thus, it follows from Mayer-Vietoris sequence (for relative $S^1$-equivariant cohomology) that the map
    \[
    (i_n)^\ast\colon H^k_{S^1}(\wt{X},\mathrm{Im}(s))\rightarrow H^k_{S^1}(\wt{X}_n,\mathrm{Im}(s))
    \]
    induced by the inclusion $i_n\colon \wt{X}_n \subset \wt{X}$ is an isomorphism for any $k<n$. Following the same logic, the following map would also be an isomorphism:
    \[
    (i_n)^\ast \colon H^k_{S^1}(\wt{X}^{S^1},\mathrm{Im}(s))\rightarrow H^k_{S^1}(\wt{X}^{S^1}_n,\mathrm{Im}(s)).
    \]
    
    Now let $k=d_j(X)$ and fix a very large integer $n>0$ such that the above two maps become isomorphisms. Then there exists a cohomology class $x\in H^k(\wt{X},\mathrm{Im}(s))$ such that $i^\ast(x)=U^i S^j \pmod{\langle S^{j+1},RS^j \rangle}$. Then we consider following commutative square.
    \[
    \xymatrix{
    H^k_{S^1} (\wt{X},\mathrm{Im}(s)) \ar[r]^{(i_n)^\ast} \ar[d]^{i^\ast} & H^k_{S^1}(\wt{X}_n,\mathrm{Im}(s)) \ar[d]^{i^\ast} \\
    H^k_{S^1}(\wt{X}^{S^1},\mathrm{Im}(s)) \ar[r]^{(i_n)^\ast} & H^k_{S^1}(\wt{X}^{S^1}_n ,\mathrm{Im}(s))
    }
    \]
    Since we assumed $n$ to be sufficently large, we may assume that $k<n$, so that the maps $(i_n)^\ast$ on the top and bottom of the diagram are isomorphisms. Then we see that the cohomology class $(i_n)^\ast(x)$ also satisfies $i^\ast(x)=U^i S^j \pmod{\langle S^{j+1},RS^j \rangle}$, and thus we get
    \[
    d_{j,n}(\wt{X}_n) \le k = d_j(X).
    \]
    Now denote $k^\prime = d_{j,n}(\wt{X}_n)$; we know that $k^\prime \le k < n$, so we also have the following commutative square.
    \[
    \xymatrix{
    H^{k^\prime}_{S^1} (\wt{X},\mathrm{Im}(s)) \ar[r]^{(i_n)^\ast} \ar[d]^{i^\ast} & H^{k^\prime}_{S^1}(\wt{X}_n,\mathrm{Im}(s)) \ar[d]^{i^\ast} \\
    H^{k^\prime}_{S^1}(\wt{X}^{S^1},\mathrm{Im}(s)) \ar[r]^{(i_n)^\ast} & H^{k^\prime}_{S^1}(\wt{X}^{S^1}_n ,\mathrm{Im}(s))
    }
    \]
    By assumption, there exist a cohomology class $x^\prime \in H^{k^\prime}_{S^1}(\wt{X}_n,\mathrm{Im}(s))$ satisfying $i^\ast(x^\prime)=U^i S^j \pmod{\langle S^{j+1},RS^j \rangle}$. Since $(i_n)^\ast$ is an isomorphism, we can consider the cohomology class $x^\prime_0((i_n)^\ast)^{-1}(x^\prime)$; it also satisfies $i^\ast(x^\prime_0)=U^i S^j \pmod{\langle S^{j+1},RS^j \rangle}$, and thus we get
    \[
    d_j(X) \le k^\prime = d_{j,n}(\wt{X}_n).
    \]
    Therefore we deduce that $d_{j,n}(\wt{X}_n)=d_j(X)$ for all sufficiently large $n$, as desired.
\end{proof}

It follows from \Cref{prop: stability} that for any $j\ge 0$, we have 
\[
\lim_{n\rightarrow\infty} d_{j,n}(\wt{X}_n) =d_j(\wt{X})=d_j(X).
\]
Now, we consider a $\Z_p$-action on a rational homology 3-sphere $Y$ that preserves a spin$^c$ structure $\mathfrak{t}$ on $Y$. Take $\Z_p$-invariant Riemann metric $g_Y$.  Then we have a space of family SWF defined by 
\[
{\bf I}^\nu _\lambda (Y) := E\Z_p \times_{\Z_p} I^\nu _\lambda (Y) \to B\Z_p. 
\]
Then, we define a family Fr\o yshov invariant by 
\[
{\bf d}_{j,n} ( Y, \mathfrak{t} ) := d_{j,n}  ({\bf I}^\nu _\lambda (Y)_n) - 2 n (Y, \mathfrak{t} , g_Y)+  \dim V^0_\lambda(g_Y), 
\]
where ${\bf I}^\nu _\lambda (Y)_n$ is the restriction of  ${\bf I}^\nu _\lambda (Y)$ to the subspace $(B\Z_p)_n$.
One can see ${\bf d}_{j,n} ( Y, \mathfrak{t} )$ is independent of the choices of $\nu$, $\lambda$, and $g_Y $ from \cref{lem:suspension_thom} of $|\nu|$, $|\lambda|$ are sufficiently large as it is checked in \cite{Baraglia-Hekmati:2024-1}. \cref{prop: stability}  immediately gives the following corollary:
\begin{cor}\label{equality} Let $Y$ be a rational homology sphere with a $\mathbb{Z}_p$-action, and let $\frakt$ be a $\Z_p$-invariant spin$^c$ structure.  For any $j\in \Z_{\geq 0}$, there exists $N_j>0$ such that for any $n \geq N_j$,  we have 
    \[\pushQED{\qed}
   \frac{1}{2}{\bf d}_{j,n} ( Y, \mathfrak{t} ) = \frac{1}{2}{ d}_{j} ( Y, \mathfrak{t} ) =  \delta_j^{(p)} (Y, \frakt).  \qedhere
    \]
\end{cor}

% \begin{proof}
%     We put 
%     \[
%    \wt{X}_n =  {\bf I}^\nu _\lambda (Y) \cap \pi^{-1} ( (BG)_n)  \subset \wt{X} = {\bf I}^\nu _\lambda (Y)
%     \]
%     which are regarded as spaces of family SWF's, where $\pi$ denote the projection of the Borel construction $E\Z_p \times_{\Z_p} I^\nu _\lambda (Y) \to B\Z_p$. By applying \cref{prop: stability} to them, one can see that 
%     for each $j \geq 0$, there exists a constant $C_j$ such that for any $n \geq C_j$, we have 
%    \[
%      d_{j,n}(\tilde{X}_n) = d_j(\tilde{X}) =  d_j({X}).
%     \]
%     This gives the conclusion. 
% \end{proof}

\subsection{Fr\o yshov type inequality from homotopy coherent action}\label{Fryshov type inequality from homotopy coherent action}

The following is the general theorem in this paper. 

\begin{thm}\label{Froyshov ineq}
Let $(X, \mathfrak{s})$ be a $4$-dimensional negative-definite spin$^c$ cobordism from a rational homology $3$-sphere $(Y,\mathfrak{t})$ to the empty set. 
If we have a homotopy coherent $\Z_p$ action on $X$ whose restriction to the boundaries are strict action and its homology monodromy is trivial, then we have 
    \[
    \delta_j^{(p)} (Y, \mathfrak{t}) + \frac{c_1(\mathfrak{s} )^2 - \sigma (X)}{8} \leq 0 
    \]
    for any $j\in \Z_{\geq 0}$. %\textcolor{red}{SK: do we need trivial homology monodromy condition here? MT: Yes, I added } 
\end{thm}

\begin{proof}

Since the monodromy of the coherent $\Z_p$-action is trivial, we have a lifting map 
\[
B\Z_p \to B \mathrm{Aut} (X, \mathfrak{s})
\]
from \cref{existence of lift}. 
We apply \cref{lem:ineq} and \cref{lem:suspension_thom} to the fiberwise continuous map obtained in \cref{hc Bf inv} with a dual description \eqref{dagger version} and obtain 
\[
   \frac{1}{2} {\bf d}_{j,n} ( Y, \mathfrak{t} )    + \frac{c_1(\mathfrak{s} )^2 - \sigma (X)}{8} \leq 0
\]
Now, we take the limit $n\rightarrow\infty$ and use \Cref{equality} to get $    \frac{1}{2}  {\bf d}_{j,n} ( Y, \mathfrak{t} )   =    \delta_j^{(p)} (Y, \mathfrak{t})  $. The theorem follows.
%\textcolor{red}{Can you write a bit more, using \Cref{prop: stability}?}
\end{proof}

\section{Proof of main result}\label{Proof of main result}

% (generated by a diffeomorphism $\tau$)
The main result follows as an application of \cref{Froyshov ineq} combined with the computations of equivariant Fr\o yshov invariants done in \cite{Baraglia-Hekmati:2022-1}. We begin with the following lemma: 
\begin{lem}\label{lem:positive}
    If $M$ is a homology 3-sphere that bounds a smooth, positive-definite 4-manifold $W_+$ with $b_1(W_+) = 0$, then $\delta(M) \leq 0$. Furthermore, for any prime $p$, if $M$ is endowed with a free smooth $\mathbb{Z}_p$-action  which extends to a smooth homotopy coherent $\mathbb{Z}_p$-action on $W_+$, which has trivial homology monodromy, then $\delta^{(p)}_j(M)\le 0$ for all $j\ge 0$.
\end{lem}
\begin{proof}
    The $4$-manifold $W_+$ can be seen as a negative-definite cobordism from $M$ to the empty set, and any $\mathrm{spin}^c$ structure $\mathfrak{s}$ on $W_+$ restricts to the unique $\mathrm{spin}^c$ structure on $M$. Hence we get
    \[
   \delta(M) + \frac{c_1(\mathfrak{s})^2+b_2(W_+)}{8} \le    0
    \]
    which is the original Fr\o yshov inequality (see \cite{Fr96, Ma03} for example).  Since $H_1(M;\mathbb{Z})=0$, the intersection form of $W_+$ is unimodular, so it follows from a result of Elkies \cite{elkies1995characterization} that there exists a $\mathrm{spin}^c$ structure $\mathfrak{s}$ on $W_+$ satisfying $c_1(\mathfrak{s})^2 \ge -b_2(W_+)$. Thus we deduce that $\delta(M)\le 0$.

    In the presence of $\mathbb{Z}_p$-actions, we again view $W_+$ as a negative-definite cobordism from $M$ to the empty set, and apply \Cref{Froyshov ineq} to see that $\delta^{(p)}_j(M)\leq 0$ for all $j\ge 0$.
\end{proof}

By dualizing the arguments in the proof of \Cref{lem:positive}, we get the following lemma.

\begin{lem} \label{lem:negative}
If $M$ is a homology 3-sphere that bounds a smooth, negative-definite 4-manifold $W_-$ with $b_1(W_-) = 0$, then $\delta(M) \geq 0$.\qed
\end{lem}

% \begin{proof}
%     The proof is almost identical to \Cref{lem:positive}.  
% \end{proof}

Now, we give a proof of \cref{thm:main}. 

\begin{proof}[Proof of \cref{thm:main}] Let $Y \neq S^3$ be a Brieskorn homology sphere, written by $Y= \Sigma(a_1, \ldots, a_n)$ with the orientation which naturally has a negative-definite resolution. 

Fix a large enough odd prime $p$ so that  $p$ does not divide $a_1 a_2 \cdots a_n$. Baraglia--Hekmati's equivariant Fr\o yshov invariants $\delta^{(p)}_j (Y)$ for the subgroup $\Z_p$ of the Seifert $S^1$-action are computed in \cite[Theorem~1.4]{Baraglia-Hekmati:2022-1} as
\begin{align}\label{eq Fr calculation}
\delta^{(p)}_\infty (Y) -\delta(Y)   = \rank HF^{red}(Y) - \rank HF^{red}(Y/\Z_p; \frakt_0)
\end{align}
for any spin$^c$ structure $ \frakt_0$ on $Y/\Z_p$, where $HF^{red}(Y)$ denotes the reduced Heegaard Floer homology.
Since we are assuming that $Y$ bounds a smooth, compact, positive-definite 4-manifold $X$ with $b_1(X) = 0$, it follows from \Cref{lem:positive} and \Cref{lem:negative} that $\delta(Y)=0$. Moreover, Baraglia--Hekmati~\cite[Theorem~1.5]{Baraglia-Hekmati:2022-1} proved that the right-hand side of \eqref{eq Fr calculation} is positive except for $\Sigma(2,3,5)$ and $\Sigma(2,3,11)$. In both cases, we have $\delta=1$, so they do not bound a positive-definite 4-manifold and hence can be ignored. Thus, we see 
\[
\delta^{(p)}_\infty (Y) = \delta^{(p)}_j(Y) >0 
\]
for a sufficiently large $j$. 

Now we will prove that the boundary Dehn twist along $Y$  has infinite order in $\pi_0(\mathrm{Diff}^+(X,\partial X))$. Suppose on the contrary that it has finite order. Then, from \Cref{extensiononcoborodism} and \Cref{rem:higherpower}, we conclude that for a sufficiently large prime $p$, there exists a homotopy coherent $\Z_p$-action on $X$ that extends the $\Z_p$-action on $Y$, with trivial homotopy monodromy. In particular, its homology monodromy is also trivial. Hence it follows from \Cref{lem:positive} that $\delta^{(p)}_j (Y) \le 0$, a contradiction.
\end{proof}

% Now we will prove that the boundary Dehn twist along $Y$  has infinite order in $\pi_0(\mathrm{Diff}^+(X,\partial X))$. Suppose on the contrary that it has finite order. Then, from \cref{lem:extensiononcylinder}, \cref{extensiononcoborodism}, and \Cref{rem:higherpower}, we conclude that for a sufficiently large prime number $p$, there is a homotopy coherent $\Z_p$-action on $X$ extending the $\Z_p$-action on $Y$. Moreover, we know from \Cref{lem:monodromy_trivial} that its homotopy-monodromy is trivial, which implies that its homology monodromy is also trivial. Hence it follows from \Cref{lem:positive} that $\delta^{(p)}_j (Y) \le 0$, a contradiction. The theorem follows.

\bibliographystyle{alpha}
\bibliography{tex}

\end{document}